\pdfoutput=1
\documentclass[onefignum,onetabnum]{siamart171218}

\usepackage{amsfonts}
\usepackage{graphicx}
\usepackage{epstopdf}
\usepackage{algorithmic}
\ifpdf
  \DeclareGraphicsExtensions{.eps,.pdf,.png,.jpg}
\else
  \DeclareGraphicsExtensions{.eps}
\fi

\newsiamremark{remark}{Remark}
\newsiamremark{hypothesis}{Hypothesis}
\crefname{hypothesis}{Hypothesis}{Hypotheses}
\newsiamthm{claim}{Claim}

\headers{\phantom{a}}{Arnel I. Smith, Elly Do, and Chao Chen}

\title{Adaptive, Matrix-Free Low-Rank Approximation}

\author{Arnel I. Smith\thanks{Department of  Mathematics, North Carolina State University, Raleigh, NC
  (\email{aismith@ncsu.edu}).}
\and Elly Do\thanks{Department of  Mathematics, North Carolina State University, Raleigh, NC
  (\email{ntdo2@ncsu.edu}).}
\and Chao Chen\thanks{Department of  Mathematics, North Carolina State University, Raleigh, NC
  (\email{cchen49@ncsu.edu}).}
}
\usepackage{amsopn}

\ifpdf
\hypersetup{
  pdftitle={Adaptive, Matrix-Free Low-Rank Approximation},
  pdfauthor={Arnel I. Smith, Elly Do, and Chao Chen}
}
\fi

\usepackage{amssymb}
\usepackage{epstopdf}
\usepackage{algorithm}
\usepackage{booktabs}
\usepackage{graphicx}
\usepackage{caption}
\usepackage{tikz}
\usepackage{subcaption}
\usepackage{float}
\usepackage{placeins}

\newcommand\vx{\mathbf{x}}
\newcommand\vy{\mathbf{y}}

\newcommand\ma{\mathbf{A}}
\newcommand\mb{\mathbf{B}}

\newcommand\me{\mathbf{E}}

\newcommand\mO{\mathbf{\Omega}}

\newcommand\mq{\mathbf{Q}}

\newcommand\ms{\mathbf{S}}

\newcommand\mU{\mathbf{U}}
\newcommand\mv{\mathbf{V}}

\newcommand\my{\mathbf{Y}}
\newcommand\mz{\mathbf{Z}}

\usepackage{xcolor}

\newcommand{\algcomment}[1]{\hfill$\triangleright$~#1}

\graphicspath{ {./FIGURES/} }

\begin{document}

\maketitle
\begin{abstract}
We study fixed-tolerance low-rank approximation in the matrix-free setting, where a matrix or linear operator $\mathbf{A}$ is accessible only through matrix-vector products and its rank must be determined adaptively to meet a prescribed error tolerance. We introduce a family of adaptive, matrix-free randomized QB algorithms. A randomized error indicator estimates the residual norm---in either the Frobenius or the spectral norm---directly from a random sketch, remaining accurate down to machine precision. A matrix-free rank-pruning step decouples the computational block size from the final rank, so that large, BLAS-3-friendly blocks can be used without over-estimating the rank, and an adjoint-free variant returns the orthonormal basis using only the forward operator. Across test matrices with diverse singular-value decays, the proposed methods attain ranks close to the truncated-SVD optimum while meeting the prescribed tolerance with high probability.

\end{abstract}



\section{Introduction}

We consider the \textbf{fixed-tolerance low-rank approximation} problem. Given a matrix $\mathbf{A} \in \mathbb{R}^{m \times n}$ and a target tolerance $\varepsilon > 0$, we seek to compute matrices $\mathbf{Q} \in \mathbb{R}^{m \times k}$ and $\mathbf{B} \in \mathbb{R}^{k \times n}$ such that
\begin{equation} \label{eq:auv}
\| \mathbf{A} - \mathbf{Q} \mathbf{B} \| \le \tau,
\end{equation}
where the rank $k \ll \min(m,n)$ is \emph{a priori} unknown and must be determined dynamically. In \cref{eq:auv}, the tolerance $\tau$ is a user-prescribed parameter, specified either as an absolute threshold $\tau = \varepsilon$ or a relative threshold $\tau = \varepsilon \, \|\ma\|$. We focus on the Frobenius norm $\|\cdot\|_\text{F}$ and the spectral norm $\|\cdot\|_2$. As is typical of randomized methods, the algorithms we develop satisfy \eqref{eq:auv} with high probability and, in the spectral norm, up to a modest constant (\cref{thm:pruning_spec}).

We operate in the \textbf{matrix-free setting}: $\mathbf{A}$ is accessible only through matrix-vector products (matvecs)---the forward and adjoint actions $\mathbf{x} \mapsto \mathbf{A}\mathbf{x}$ and $\mathbf{y} \mapsto \mathbf{A}^\top\mathbf{y}$---while its individual entries are not. This setting is ubiquitous in computational science. In Bayesian inverse problems and uncertainty quantification, for instance, the Hessian of the log-posterior is defined implicitly, yet a low-rank approximation of it accelerates downstream optimization and sampling~\cite{bui2013computational, flath2011fast}. There, a single matvec may require a full forward and adjoint PDE solve, so the dominant cost is the number of operator applications, not the internal dense linear algebra. \textbf{Our goal is therefore to solve \eqref{eq:auv} while minimizing the number of matvecs, and to do so accurately for $\varepsilon$ down to machine precision.}

We focus on a specific structural form of \eqref{eq:auv}, the \textbf{QB approximation}~\cite{MVO,GULI}, where $\mathbf{Q}$ has orthonormal columns; the choice $\mathbf{B} = \mathbf{Q}^\top\mathbf{A}$ then minimizes the residual in the Frobenius norm. QB is a building block for many matrix factorizations, decoupling range identification (the range finder) from entry approximation (the projection): examples include the truncated SVD~\cite{GoVa13}, the interpolative decomposition~\cite{cheng2005compression}, the higher-order SVD for tensor compression~\cite{de2000multilinear}, and hierarchical matrix approximations~\cite{hackbusch2015hierarchical, martinsson2016compressing}. Each relies on the fixed-tolerance QB primitive that the matrix-free methods of \cref{sec: 3} and \cref{sec: 4} provide.

\subsection{Existing Work and Limitations}

\paragraph{Classical deterministic methods}
Direct factorization methods such as the singular value decomposition (SVD) and rank-revealing QR (RRQR)~\cite{golub2013matrix, gu1996efficient} offer excellent stability and deterministic error bounds, but are designed for explicitly stored matrices: they require entry-wise access and scale as $\mathcal{O}(\min(m,n)mn)$, so they are inapplicable in the matrix-free setting.

Iterative Krylov subspace methods such as Lanczos and Arnoldi~\cite{saad2011numerical} use only matvecs to extract the leading singular triplets. Although effective for rapidly decaying spectra, they are numerically unstable in finite precision---the Lanczos vectors lose orthogonality, requiring expensive re-orthogonalization~\cite{paige1971computation}---and converge slowly for flat or clustered spectra, requiring many sequential matvecs.

Cross and skeleton methods such as the adaptive cross approximation (ACA)~\cite{bebendorf2000approximation, bebendorf2008hierarchical} build skeleton decompositions from a sublinear number of rows and columns. Although scalable for smooth kernels (e.g., boundary element methods), ACA must evaluate individual entries $A_{ij}$ at arbitrary positions; for implicit operators such as Hessians, a single entry costs a full matvec, negating this advantage. Its greedy pivoting is also prone to instability or premature convergence on matrices with zero blocks, noise, or localized features~\cite{heldring2016accuracy}.

\paragraph{Modern randomized methods}
Randomized numerical linear algebra now provides a mature toolkit for low-rank approximation~\cite{martinsson2020randomized, pearce2025randomized}. Randomized rank-revealing factorizations, such as randomized column-pivoted QR~\cite{martinsson2017householder, duersch2017randomized}, use random embeddings to select pivot columns efficiently, bypassing the communication overhead of classical pivoting. However, they still update the trailing matrix explicitly, so they cannot operate in a matrix-free setting.

Fixed-rank matrix-free methods such as the randomized SVD (RSVD) \cite{halko2011finding}, the Nyström method for positive semi-definite matrices \cite{gittens2013revisiting}, and the generalized Nyström scheme \cite{nakatsukasa2020fast} are matrix-free and parallelizable, but \textbf{fixed-rank}: the user must supply the target rank \emph{a priori} (which can itself be estimated by randomized numerical-rank estimation~\cite{meier2024fast}). Pass-efficient variants further reduce the number of views (passes) over $\ma$~\cite{bjarkason2018pass}, but they too require the rank in advance. In multiscale applications such as $\mathcal{H}$-matrix construction~\cite{chandrasekaran2010numerical,xia2013randomized,hackbusch2015hierarchical,ghysels2016robust}, where the numerical rank fluctuates across thousands of off-diagonal blocks, a fixed rank causes either memory over-allocation or loss of fidelity.

To address adaptivity, Halko, Martinsson, and Tropp~\cite[Algorithm 4.2]{halko2011finding} introduced an adaptive range finder, hereafter \texttt{randQB\_HMT}, the state-of-the-art (SOTA) adaptive method in the spectral norm; it is based on a probabilistic error bound that is pessimistic in practice and over-estimates the rank.
Yu, Gu, and Li~\cite{GULI} introduced \texttt{randQB\_EI}, the SOTA adaptive method in the Frobenius norm, which builds the range block by block and tracks the residual error dynamically. However, it estimates the residual Frobenius norm by Gramian-type tracking ($\mathbf{A}\mathbf{A}^\top$), squaring the effective condition number and capping the attainable precision near $\sqrt{\varepsilon_{\text{mach}}}$ ($\approx 10^{-8}$ in IEEE double precision). Because its indicator subtracts $\|\mathbf{B}\|_{\mathrm{F}}^2$ from $\|\mathbf{A}\|_{\mathrm{F}}^2$, it also requires the Frobenius norm $\|\mathbf{A}\|_{\mathrm{F}}$ to be computed explicitly from $\mathbf{A}$ and is thus not matrix-free.

A related adaptive Frobenius-norm method, by Gorman et al.~\cite{gorman2019robust}, attains comparable accuracy but certifies the error by post-processing the random sketches from all iterations, making it substantially more expensive than the rank pruning introduced here (\cref{sec:rr}). More recently, Liu and Yu~\cite{liu2025efficient} proposed adaptive, fixed-threshold randomized algorithms with error analysis in both the Frobenius and spectral norms; like the methods above, however, they operate on the explicit entries of $\ma$ and are not matrix-free. Adaptive randomized schemes have likewise been developed for other factorizations and settings, including rank-adaptive CUR~\cite{pritchard2025fast}, adaptive low-rank approximation for image compression~\cite{xu2025novel}, and dynamical low-rank approximation~\cite{carrel2024randomized}.

In short, no existing method is simultaneously matrix-free, adaptive in rank, and accurate down to machine precision in both the Frobenius and spectral norms; closing this gap is the goal of the present work.

\subsection{Contributions}

The primary contribution is a family of adaptive, matrix-free QB algorithms that solve the fixed-tolerance low-rank approximation problem down to machine precision ($\varepsilon_{\text{mach}}$) in both the Frobenius and spectral norms. They build on the blocked randomized QB iteration~\cite{MVO,GULI}, which constructs $\mathbf{Q}$ and $\mathbf{B}$ incrementally in blocks of $b$ columns per iteration. Our main developments are:
\begin{itemize}

\item
\textbf{Randomized residual indicators:}
Rather than a single, generic indicator, we develop a family of randomized error indicators tailored to each method---the sketched residual norm in the Frobenius case and a look-ahead block norm in the spectral case, each with an adjoint-free counterpart. Every indicator is read off the block sketch already computed at each iteration, at negligible cost, and avoids the failure modes of prior adaptive schemes: the ``precision wall'' near $\sqrt{\varepsilon_{\text{mach}}}$ of energy-subtraction (Gramian) indicators~\cite{GULI}, and the rank over-estimation caused by the conservative probabilistic bound of~\cite{halko2011finding}, remaining accurate and tight down to machine precision.

\item
\textbf{Rank Pruning and the Block-Size Dilemma:}
Our most significant algorithmic innovation is a post-iteration rank pruning step. Large block sizes $b$ are needed for matrix-matrix (BLAS-3) efficiency, but they over-estimate the rank when the tolerance is met mid-block. Our pruning decouples the block size from the final rank, preserving high performance while retaining a near-minimal rank.

\item
\textbf{Adjoint-Free Adaptivity:}
We propose adjoint-free variants that determine the rank without the adjoint $\mathbf{A}^\top$. When the adjoint is expensive or unavailable and only the basis $\mathbf{Q}$ is needed, these variants roughly halve the passes over the data ($\mathbf{B}$ is not formed).

\end{itemize}

We emphasize that the basic matrix-free and adjoint-free blocked QB frameworks have appeared in prior work; our contributions are the rank-pruning step, the matrix-free estimate of $\|\ma\|$ for relative tolerances, and a thorough numerical study of the resulting methods.

\Cref{tab:methods} summarizes the four resulting methods, organized by norm and by whether the adjoint operator is required.

\begin{table}[htbp]
\centering
\caption{The four proposed adaptive, matrix-free QB methods. The matrix-free (\texttt{MF}) variants use both the forward and adjoint actions ($\ma$ and $\ma^\top$); the adjoint-free (\texttt{AF}) variants use only the forward action $\ma$.}
\label{tab:methods}
\begin{tabular}{lcc}
\toprule
 & Matrix-free (\texttt{MF}) & Adjoint-free (\texttt{AF}) \\
\midrule
Frobenius norm & \texttt{randQB\_MF\_Fro}~(Alg.~\ref{alg:adapt_fro}) & \texttt{randQB\_AF\_Fro}~(Alg.~\ref{alg:adjoint_free_frob}) \\
Spectral norm & \texttt{randQB\_MF\_Spec}~(Alg.~\ref{alg:adapt_spec}) & \texttt{randQB\_AF\_Spec}~(Alg.~\ref{alg:adjoint_free_spec}) \\
\bottomrule
\end{tabular}
\end{table}

\section{Preliminaries and Problem Formulation}

\subsection{Notation}

Throughout this work, matrices are denoted by bold uppercase letters (e.g., $\mathbf{A}, \mathbf{B}$), vectors by bold lowercase letters (e.g., $\mathbf{x}, \mathbf{y}$), and scalars by standard lowercase or Greek letters (e.g., $k, \varepsilon$). The singular value decomposition (SVD) of $\mathbf{A}$ is denoted as $\mathbf{A} = \mathbf{U} \mathbf{\Sigma} \mathbf{V}^\top$, with singular values $\sigma_1 \ge \sigma_2 \ge \dots \ge 0$.
We utilize two standard matrix norms to measure the approximation error:
\begin{itemize}
\item
The Frobenius norm, defined as $\|\mathbf{A}\|_{\text{F}} = \sqrt{\sum_{i,j} |A_{ij}|^2} = \sqrt{\sum_j \sigma_j^2}$.

\item
The spectral norm, defined as $\|\mathbf{A}\|_2 = \sigma_1$.
\end{itemize}

Our pseudocode uses two MATLAB-style primitives: $\texttt{randn}(n, b)$ returns an $n \times b$ matrix with independent standard Gaussian $\mathcal{N}(0,1)$ entries, and $\texttt{orth}(\mathbf{Y})$ returns a matrix whose columns form an orthonormal basis for the range of $\mathbf{Y} \in \mathbb{R}^{n \times b}$ with $b < n$, computed in practice via a thin (unpivoted) QR factorization.

\begin{remark}[Complex matrices]
For clarity, all results are stated for real matrices. Every algorithm and result carries over verbatim to complex matrices by replacing each matrix transpose with the conjugate transpose. \end{remark}

\subsection{The Standard Randomized QB Decomposition}
\label{sec:qb}
The standard randomized algorithm for computing a low-rank QB approximation proceeds in two primary phases: range finding and projection. The method comes with a priori error guarantees: the approximation error is within a polynomial factor of the optimal value $\sigma_{k+1}(\ma)$ with high probability~\cite{halko2011finding}; see~\cite{scotto2024unified} for a unified error analysis in both the Frobenius and spectral norms. It does, however, require the target rank $k$ to be fixed in advance. The randomized QB method goes as follows:

\begin{enumerate}
\item
\textbf{Range Finding:}
 Draw a Gaussian random test matrix $\mathbf{\Omega} \in \mathbb{R}^{n \times (k+p)}$ (Here, $p$ is an oversampling parameter chosen typically as 5 or 10). Compute the sketch $\mathbf{Y} = \mathbf{A}\mathbf{\Omega}$ and form an orthonormal basis for its range: $\mathbf{Q} = \texttt{orth}(\mathbf{Y})$.

\item
\textbf{Projection:}
Project the original matrix onto this basis to form the core matrix: $\mathbf{B} = \mathbf{Q}^\top\mathbf{A}$.

\end{enumerate}

For adaptive, fixed-tolerance settings, the blocked QB method~\cite{MVO,GULI} builds $\mathbf{Q}$ and $\mathbf{B}$ incrementally in blocks of size $b$. At the $i$-th iteration, let $\mathbf{Q}^{(i-1)}$ and $\mathbf{B}^{(i-1)}$ be the accumulated bases. The algorithm sketches the residual operator $\mathbf{E}_{i-1} = \mathbf{A} - \mathbf{Q}^{(i-1)}\mathbf{B}^{(i-1)}$ as $\mathbf{Y}_i = \mathbf{E}_{i-1} \mathbf{\Omega}_i$ using a block of random vectors $\mathbf{\Omega}_i \in \mathbb{R}^{n \times b}$.
Once $\mathbf{Q}_i = \texttt{orth}(\mathbf{Y}_i)$ is computed, the block $\mathbf{B}_i$ is formed via $\mathbf{B}_i = \mathbf{Q}_i^\top\mathbf{A}$. This process is summarized in \cref{alg:blocked_QB}

\begin{algorithm}

\caption{Adaptive blocked QB iteration (variant of \cite{MVO,GULI})}
\label{alg:blocked_QB}

\begin{algorithmic}[1]
\REQUIRE Matrix $\mathbf{A}\in \mathbb{R}^{m \times n }$, (absolute) tolerance $\varepsilon$, and block size $b.$
\ENSURE Matrix $\mq$ with orthonormal columns and matrix $\mb$ such that $\|\ma-\mq\mb\| \leq \varepsilon$ with high probability.

\item[]

\STATE $\me_0 = \ma$
\FOR {$i=1,2,\ldots$}

    \STATE $\mathbf{\Omega}_i \gets \texttt{randn}(n, b)$ \algcomment{Gaussian random block}
    \STATE $\my_i \gets \me_{i-1}\mO_i$
    \algcomment{Sample range space of the residual matrix }
    \STATE $\mq_{i} \gets \texttt{orth}(\my_i)$
    \algcomment{Thin QR decomposition}
    \STATE $\mb_{i} \gets \mq_{i}^\top\ma$

    \STATE $\me_{i} \gets \me_{i-1} - \mq_{i}\mb_{i}$ \algcomment{Update residual matrix $\me_{i} = \ma - \sum_{k = 1}^{i}\mq_k\mb_k$}
    \IF{$\|\me_i\| \leq \varepsilon$}
        \STATE \textbf{break} \algcomment{Target tolerance reached}
    \ENDIF
\ENDFOR

\STATE $\mq \gets [\mq_1\; \mq_2 \; \cdots\; \mq_{i} ]$

\STATE $\mb \gets [\mb_1^\top\; \mb_2^\top \; \cdots\; \mb_{i}^\top ]^\top$

\end{algorithmic}
\end{algorithm}

Notice that \cref{alg:blocked_QB} measures the error at each iteration via the (Frobenius or spectral) norm of the residual matrix $\me_i$. Forming $\me_i$ explicitly violates the matrix-free constraint: it costs as many matvecs as $\ma$ has rows or columns.

\subsection{Estimating the Frobenius Norm}
We recall a classical result on randomized Frobenius-norm estimation, due to Hutchinson~\cite{Hutchinson}, which underlies our residual tracking without forming matrix entries.

\begin{theorem}[Lemma 1 in \cite{Hutchinson}] \label{thm:expect thm - Y_i}
    Let $\ma \in \mathbb{R}^{m \times n}$ and let \(\mO_i \in \mathbb{R}^{n \times b}\) be a random Gaussian matrix whose entries are independent and identically distributed normal random variables with mean zero and variance $1$. Then
    \begin{align*}
        \mathbb{E}\left[\frac{1}{b} \|\ma\mO_i\|_{\mathrm{F}}^2\right] = \|\ma\|_{\mathrm{F}}^2.
    \end{align*}
\end{theorem}

\Cref{thm:expect thm - Y_i} establishes unbiasedness in expectation; concentration bounds~\cite{Frobenius} further guarantee that a single realization stays close to the true norm with high probability. For an order-of-magnitude estimate of the residual error, a small block size ($b = 5$ or $10$) suffices, so convergence can be monitored with cheap, low-dimensional sketches.

\subsection{Estimating the Spectral Norm}
For stopping criteria in the spectral norm, we must track the singular values of the residual (and of $\ma$) from sketched information alone. We review a bound relating the singular values of the projection $\mb = \mq^\top\ma$ to those of $\ma$ in the QB framework of \cref{sec:qb}.

\begin{theorem}[Theorem 9 in \cite{Arvind_singular_value}] \label{thm:arvind}
Let $\mb = \mq^\top \ma$ be the projection produced by the randomized QB method with a Gaussian sketch $\mO \in \mathbb{R}^{n \times (k+p)}$, $k+p \le \min(m,n)$. Then, for $i = 1, \ldots, k$,
\[
\sigma_i(\ma) \ge \sigma_i(\mb) \ge \sigma_i(\ma)/\eta_i,
\]
where each factor $\eta_i \ge 1$ is bounded by a constant depending only on $k$, $p$, and a prescribed failure probability, with high probability~\cite[Theorem 5.8]{gu2015subspace}.
\end{theorem}

By \cref{thm:arvind}, the spectral norm of the projection $\mb$ is a reliable proxy for the dominant singular values of $\ma$: with $i=1$, $\|\ma\|_2 = \sigma_1(\ma) \le \eta_1 \sigma_1(\mb) = \eta_1 \|\mb\|_2$. Applying this with the residual $\me_i = \ma - \mq^{(i)}\mb^{(i)}$ in place of $\ma$ lets us monitor the spectral-norm error without ever forming the residual.

A second, complementary tool is the notion of a \emph{subspace embedding}~\cite{martinsson2020randomized}: a sketch that preserves norms on a subspace also preserves its singular values. The simplest such embedding is a (scaled) random Gaussian matrix~\cite{martinsson2020randomized}, which is the sketch used throughout this work.

\begin{lemma}[Subspace embedding and spectral approximation] \label{lem:subspace_embedding}
Let $\ma \in \mathbb{R}^{m \times n}$ and let $\ms \in \mathbb{R}^{s \times m}$ be an $\varepsilon$-subspace embedding for the range of $\ma$: for some $0 < \varepsilon < 1$,
\begin{align*}
(1-\varepsilon)\,\|\ma\vx\|_2 \le \|\ms\ma\vx\|_2 \le (1+\varepsilon)\,\|\ma\vx\|_2 \qquad \text{for all } \vx \in \mathbb{R}^n.
\end{align*}
Then $\tilde{\ma} := \ms\ma$ is an $\varepsilon$-spectral approximation of $\ma$: it preserves $\ma^\top\ma$ to relative error $\varepsilon$,
\begin{align*}
(1-\varepsilon)^2\,\ma^\top\ma \preceq \tilde{\ma}^\top\tilde{\ma} \preceq (1+\varepsilon)^2\,\ma^\top\ma,
\end{align*}
and in particular $(1-\varepsilon)\,\sigma_i(\ma) \le \sigma_i(\tilde{\ma}) \le (1+\varepsilon)\,\sigma_i(\ma)$ for every singular value.
\end{lemma}
\begin{proof}
The embedding condition squared reads $(1-\varepsilon)^2\,\vx^\top\ma^\top\ma\vx \le \vx^\top\tilde{\ma}^\top\tilde{\ma}\vx \le (1+\varepsilon)^2\,\vx^\top\ma^\top\ma\vx$ for all $\vx$, which is the stated Loewner ordering. The singular-value bounds follow from the Courant--Fischer min--max characterization.
\end{proof}

\section{Adaptive Matrix-free Algorithms in Frobenius norm} \label{sec: 3}

Suppose $\ma \in \mathbb{R}^{m \times n}$ is accessible only through its forward and adjoint actions $\vx \mapsto \ma \vx$ and $\vy \mapsto \ma^\top \vy$. Our objective is to compute $\mq \in \mathbb{R}^{m \times k}$ with orthonormal columns and $\mb = \mq^\top\ma \in \mathbb{R}^{k \times n}$ such that
\[
\|\ma - \mq\mb\|_{\mathrm{F}} \leq \tau,
\]
where the target rank $k$ is not known in advance and must be determined dynamically. The choice $\mb = \mq^\top\ma$ minimizes the Frobenius-norm error for a given basis $\mq$. The tolerance is specified either as an absolute threshold $\tau = \varepsilon$ or a relative threshold $\tau = \varepsilon \, \|\ma\|_{\mathrm{F}}$, for some $\varepsilon > 0$.

Standard approaches, such as the blocked QB algorithm~\cite{MVO}, explicitly form the residual matrix $\me_i$ at every iteration (see the residual update in \cref{alg:blocked_QB}). Consequently, they cannot be applied in the matrix-free setting. In this section, we demonstrate how to adapt the blocked QB framework into a purely matrix-free algorithm by resolving the following key algorithmic challenges:

\begin{enumerate}

\item \textbf{Residual Estimation (\cref{sec:err_ind}):} We introduce a randomized error indicator that efficiently estimates the residual norm $\|\me_i\|_{\mathrm{F}}$ \emph{without} forming $\me_i$. This bypasses the need for explicit matrix entries, enabling an adaptive stopping criterion relying exclusively on matrix-vector products.

\item \textbf{Block Size Selection (\cref{sec:rr}):} We analyze strategies for choosing the block size parameter to maximize computational efficiency while preventing severe overestimation of the numerical rank.

\item \textbf{Adjoint-Free Variant (\cref{sec:adjoint_free_frob}):} We introduce a forward-only modification that eliminates the need for adjoint queries ($\vy \mapsto \ma^\top \vy$) when only the basis $\mq$ is required.

\end{enumerate}

\subsection{Randomized Error Indicator and Matrix-Free Algorithm}
\label{sec:err_ind}

Recall the challenge in the blocked QB method is that the stopping criterion requires the error $\|\me_{i}\|_{\mathrm{F}}$, where the residual matrix
\begin{equation} \label{eq:ei}
\me_{i} = \ma - \sum_{k=1}^{i}\mq_k\mb_k \quad i=1,2,\ldots.
\end{equation}
To avoid accessing entries of $\me_{i}$, we use Hutchinson's trace (Frobenius norm) estimator. But recall that a sample matrix $\my_i = \me_{i-1} \, \mO_i$ is \emph{already computed} to build the basis at every iteration in the blocked QB algorithm. Here, $\me_{i-1}$ is the residual from the previous steps, and $\mO_i$ is a newly drawn random test matrix. Throughout, we normalize the random blocks so that $\mO_i \in \mathbb{R}^{n \times b}$ has i.i.d.~$\mathcal{N}(0, 1/b)$ entries---equivalently, $\mO_i = \texttt{randn}(n,b)/\sqrt{b}$---which absorbs the $1/\sqrt{b}$ factor that would otherwise appear in every norm estimate. Applying \cref{thm:expect thm - Y_i} to $\sqrt{b}\,\mO_i$ then dictates that
\begin{align*}
    \mathbb{E}
    \left[
        \|\my_i\|_{\mathrm{F}}^2
    \right]
        =\mathbb{E}
    \left[
        \|\me_{i-1}\mO_i\|_{\mathrm{F}}^2
    \right]
        =\|\me_{i-1}\|_{\mathrm{F}}^2.
\end{align*}
Therefore, we propose using $\|\my_i\|_{\mathrm{F}}^2$ as an unbiased randomized estimator of the current residual norm (squared) $\|\me_{i-1}\|_{\mathrm{F}}^2$ (see \cref{thm:expect thm - Y_i}).

\vspace{1em}
\fbox{\parbox{0.9\linewidth}{We use the norm of the sample matrix as a proxy for the true error:
    \begin{equation} \label{eq:est_me}
        \|\my_i\|_{\mathrm{F}}
        \approx
        \|\me_{i-1}\|_{\mathrm{F}}.
    \end{equation}
  }}
\vspace{1em}

Importantly, this estimator requires only the product $\me_{i-1} \mO_i$ and does not require explicit access to the entries of the residual. Furthermore, evaluating this estimator incurs negligible $\mathcal{O}(bm)$ additional computational cost.

We no longer need $\me_{i-1}$ for evaluating the error, but we still need to compute $\my_i = \me_{i-1} \, \mO_i$. How do we do this without storing $\me_{i-1}$? Assuming $i-1$ iterations are complete, we use a simple associativity trick:
\begin{align} \label{Y_i - full description}
\my_i
=\me_{i-1}\mO_i
& =\left(\ma-\sum_{k=1}^{i-1}\mq_k\mb_k\right)\mO_i \\
& = \underbrace{\ma\mO_i}_{\text{black-box Mat-Vec}}
-\sum_{k=1}^{i-1}\mq_k
\underbrace{(\mb_k\mO_i)}_{\text{small matrix mult}}. \notag
\end{align}
Forming $\my_i$ this way {completely eliminates} the need to ever form or store the residual matrix $\me_{i-1}$.

The combination of the randomized error indicator and the implicit residual expansion gives rise to a fully adaptive, matrix-free algorithm, summarized in \cref{alg:adapt_fro}. A defining feature of this new framework is that the sample matrix $\mathbf{Y}_i$ fulfills a dual algorithmic role. First, in its traditional capacity, it serves as the subspace generator to construct the next orthonormal block $\mathbf{Q}_i$. Second, by simply evaluating its Frobenius norm $\|\mathbf{Y}_i\|_{\mathrm{F}}$, it simultaneously provides an estimate of the current residual error.

\begin{algorithm}
\caption{\texttt{randQB\_MF\_Fro}: Adaptive, matrix-free low-rank approximation in Frobenius norm}
\label{alg:adapt_fro}
\begin{algorithmic}[1]
\REQUIRE Matrix $\mathbf{A} \in \mathbb{R}^{m \times n }$, (absolute) tolerance\footnotemark $\varepsilon > 0$, and block size $b.$
\ENSURE Matrix $\mq$ with orthonormal columns and matrix $\mb$ such that $\|\ma-\mq\mb\|_{\mathrm{F}} \leq \varepsilon$ with high probability.

\item[]

\FOR {$i=1,2,\ldots$}
\STATE $\mathbf{\Omega}_i \gets \texttt{randn}(n, b)/\sqrt{b}$ \algcomment{normalized Gaussian block}
    \STATE $\mathbf{Z}_i \gets \mathbf{A}\mathbf{\Omega}_i$ \algcomment{$\mathbf{A}$ is accessed via matvec}
    \STATE $\my_i \gets \mz_i - \sum_{k = 1}^{i-1}\mq_k\left(\mb_k\mO_i\right)$
    \algcomment{Implicit residual application via \eqref{Y_i - full description}}
    \IF{$\|\my_i\|_{\mathrm{F}} \leq \varepsilon$}
        \STATE \textbf{break} 
        \algcomment{Last computed blocks are $\mq_{i-1}$ and $\mb_{i-1}$}
    \ENDIF

    \STATE $\mathbf{Q}_i \gets \texttt{orth}(\mathbf{Y}_i)$
    \algcomment{Thin QR decomposition}
    \STATE $\mathbf{Q}_i \gets \texttt{orth}\left(\mathbf{Q}_i - \sum_{k = 1}^{i-1}\mathbf{Q}_k\left(\mathbf{Q}_k^\top\mathbf{Q}_i\right)\right)$
    \algcomment{Re-orthogonalization}
    \STATE $\mathbf{B}_i \gets \mathbf{Q}_i^\top\mathbf{A}$ \algcomment{$\mathbf{A}$ is accessed via adjoint matvec}
\ENDFOR

\STATE $\textbf{s} \gets \texttt{Prune\_Rank\_Fro}(\mb_{i-1}, \, \|\my_i\|_{\mathrm{F}},\, \varepsilon)$
\algcomment{See \cref{sec:rr}}
\STATE $\mq \gets [\mq_1 \; \ldots \; \mq_{i-2} \; \mq_{i-1}^\textbf{s}]$
\algcomment{$\mq_{i-1}^\textbf{s}$ contain the  columns subset \textbf{s} in $\mq_{i-1}$}
\STATE $\mb \gets [\mb_1^\top\; \cdots\; \mb_{i-2}^\top\; (\mb_{i-1}^\textbf{s})^\top ]^\top$
\algcomment{$\mb_{i-1}^\textbf{s}$ contain the  row subset \textbf{s} in $\mb_{i-1}$}
\end{algorithmic}
\end{algorithm}
\footnotetext{See \cref{rmk:relative tolerance} for the modification for relative tolerance.}

We now highlight two practical considerations for the robust and efficient implementation of \cref{alg:adapt_fro}.
\begin{enumerate}
\item \textbf{Loss of Orthogonality:}
In exact arithmetic, the output matrix $\mathbf{Q}$ naturally maintains orthonormal columns. In practice, finite-precision arithmetic inevitably leads to a gradual loss of orthogonality. To counteract this, a re-orthogonalization step against the previously computed basis blocks is employed after evaluating the new basis vectors.

\item \textbf{Memory Allocation:}
The iterative algorithm computes the output matrices $\mathbf{Q}$ and $\mathbf{B}$ one block at a time. Dynamically reallocating memory for these matrices with increasing sizes at every iteration incurs a severe performance penalty. An optimized implementation should instead pre-allocate memory up to a maximum expected rank. If the iterative rank exceeds this estimate, the array sizes can be doubled to amortize the cost of reallocation.

\end{enumerate}

\begin{remark}[Relative Tolerance] \label{rmk:relative tolerance}
    Thus far, we have considered an absolute tolerance, i.e., $\tau=\varepsilon$. For a relative tolerance $\tau=\varepsilon\,\|\mathbf{A}\|_{\mathrm{F}}$, we also require a matrix-free estimate of $\|\mathbf{A}\|_{\mathrm{F}}$. This makes the relative criterion more challenging than the absolute one: the algorithm must estimate $\|\mathbf{A}\|_{\mathrm{F}}$ in addition to the residual norm, without explicit access to the entries of $\mathbf{A}$. Let $\mathbf{Z}_i = \mathbf{A}\mathbf{\Omega}_i$ as computed in \cref{alg:adapt_fro}. Using the same probabilistic identity, we have
\[
\mathbb{E}
\left[ \|\mathbf{Z}_i\|_{\mathrm{F}}^2\right]
=
\|\mathbf{A}\|_{\mathrm{F}}^2.
\]
Therefore, $\|\mathbf{Z}_i\|_{\mathrm{F}}^2$ provides a randomized, matrix-free estimator of $\|\mathbf{A}\|_{\mathrm{F}}^2$. Since a new block $\mathbf{Z}_i$ is generated at every iteration, we can systematically reduce the variance of this estimate by aggregating the samples via a cumulative root-mean-square estimator:
\begin{equation} \label{eq:est_ma}
S_i
\equiv
\sqrt{\frac{1}{i} \sum_{k=1}^{i} \|\mathbf{Z}_k\|_{\mathrm{F}}^2}
\approx
\|\mathbf{A}\|_{\mathrm{F}}.
\end{equation}
Because the quantities $\mathbf{Z}_k$ are already computed during the basis generation phase, evaluating $S_i$ incurs only $\mathcal{O}(bm)$ additional computational cost at every iteration.
\end{remark}

\begin{remark}[Precision Limits] \label{rmk:randQB_ei}
    In the existing method \texttt{randQB\_EI}~\cite{GULI}, the residual norm is tracked by subtracting squared norms of the accumulated basis, i.e., $\|\me_i\|_{\mathrm{F}}^2 = \|\me_{i-1}\|_{\mathrm{F}}^2 - \|\mb_i\|_{\mathrm{F}}^2$. This approach is susceptible to catastrophic cancellation, limiting the achievable tolerance to $\mathcal{O}(\sqrt{\varepsilon_{\textbf{mach}}})$ (approximately $10^{-8}$ in IEEE double precision). Because our newly proposed error indicator evaluates the residual norm directly via implicit application, it bypasses this cancellation issue and can reliably achieve precision on the order of $\mathcal{O}(\varepsilon_{\textbf{mach}})$.
\end{remark}

\begin{remark}[Relation to prior work]
The randomized Frobenius-norm indicator and blocked iteration underlying \cref{alg:adapt_fro} are not new in themselves: a randomized (Hutchinson) Frobenius-norm estimate is used by Pearce et al.~\cite{pearce2025adaptive} to drive an adaptive interpolative decomposition, and the same matrix-free blocked QB framework underlies the hierarchically semiseparable construction of Gorman et al.~\cite{gorman2019robust}, which does not report numerical results for the basic adaptive algorithm in isolation. Our main algorithmic contribution is the rank-pruning step of \cref{sec:rr}, which decouples the block size from the final rank so that large, BLAS-3-friendly blocks can be used without over-estimating the rank; a secondary contribution is the matrix-free estimate of $\|\ma\|_{\mathrm{F}}$ for relative tolerances (\cref{rmk:relative tolerance}). We also provide a thorough numerical study (\cref{sec:numerical}) demonstrating the efficiency and machine-precision accuracy of the proposed methods.
\end{remark}

\subsection{Block Size Selection and Rank Pruning}\label{sec:rr}

Modern cache-based computing architectures are most efficient when computations can be organized as matrix-matrix products of sufficient size. For this reason, choosing a sufficiently large block size $b$ in \cref{alg:adapt_fro} can substantially improve practical performance. The drawback is that the output rank grows in multiples of $b$, meaning a large block size may overestimate the true numerical rank by as much as $b-1$, leading to unnecessary storage and downstream computational costs.

To mitigate this, we introduce a lightweight {rank pruning} procedure that prunes the final block produced by \cref{alg:adapt_fro}. The pruning operates exclusively on the last computed block (the last $b$ rows of $\mb$), and it has negligible computational overhead, just $\mathcal{O}(b n)$ operations. The objective is to keep the computed low-rank approximation as compact as possible while still rigorously satisfying the prescribed tolerance.

Suppose \cref{alg:adapt_fro} is applied to $\ma\in\mathbb{R}^{m\times n}$ and the look-ahead test passes at iteration $i \ge 2$, so the algorithm retains $\tilde{\mq} = [\mq_1\; \mq_2 \; \cdots\; \mq_{i-1} ]$ and $\tilde{\mb} = [\mb_1^\top\; \mb_2^\top \; \cdots\; \mb_{i-1}^\top ]^\top$. The state at termination reads:
    \begin{align*}
        \text{Previous step failed:}
        \quad
        &{\|\me_{i-2}\|_\mathrm{F} > \tau} \quad
        \text{where} \;\; \me_{i-2} = \ma - \sum_{k=1}^{i-2}\mq_k \mb_k \\
\text{Current step passed:}
        \quad
        &{\|\me_{i-1}\|_\mathrm{F} \le \tau}
        \quad \;\;
        \text{where} \;\; \me_{i-1} = \me_{i-2} - \mq_{i-1}\mb_{i-1}.
    \end{align*}
Here, $\mq_{i-1}\in\mathbb{R}^{m\times b}$ and $\mb_{i-1}\in\mathbb{R}^{b\times n}$ are the final blocks added to the basis. (Technically speaking, we don't have the exact errors $\|\me_{i-2}\|_\mathrm{F}$ and $\|\me_{i-1}\|_\mathrm{F}$ but their randomized estimators.)

\vspace{1em}
\fbox{\parbox{0.85\linewidth}{Since the final update block $\mq_{i-1}\mb_{i-1}$ is simply a sum of $b$ rank-$1$ outer products, our goal is to retain \textbf{as few} rank-$1$ outer products as possible to just cross the $\varepsilon$ threshold, stripping away the excess rank.
  }}
\vspace{1em}

To formalize this approach, we partition the index set $\{1,\ldots,b\}$ into two disjoint subsets:
\[
\{1,\ldots,b\} = \mathbf{s}\cup\mathbf{r}, \qquad \mathbf{s}\cap\mathbf{r}=\emptyset, \qquad |\mathbf{s}|+|\mathbf{r}|=b.
\]
Here, $\mathbf{s}$ denotes the indices to \emph{select} and $\mathbf{r}$ denotes the indices to \emph{remove}. Let $\mq_{i-1}^{\mathbf{s}}$ and $\mq_{i-1}^{\mathbf{r}}$ contain the corresponding columns in $\mq_{i-1}$, and let $\mb_{i-1}^{\mathbf{s}}$ and $\mb_{i-1}^{\mathbf{r}}$ contain the corresponding rows in $\mb_{i-1}$.
It follows that the final block update can be split as $
\mq_{i-1}\mb_{i-1} = \mq_{i-1}^{\mathbf{s}}\mb_{i-1}^{\mathbf{s}} + \mq_{i-1}^{\mathbf{r}}\mb_{i-1}^{\mathbf{r}}.$
Selecting the smallest necessary set $\mathbf{s}$ corresponds to solving
\begin{align}
\label{1st opt prob}
\min\ |\mathbf{s}| \quad\text{subject to}\quad \big\| \me_{i-2} - \mq_{i-1}^{\mathbf{s}}\mb_{i-1}^{\mathbf{s}} \big\|_{\mathrm{F}} \le \tau.
\end{align}
Once $\mathbf{s}$ is determined, the compact form of our low-rank approximation becomes ${\mq} = [\mq_1\; \cdots\; \mq_{i-2} \; \mq_{i-1}^{\mathbf{s}} ]$ and ${\mb} = [\mb_1^\top\; \cdots\; \mb_{i-2}^\top \; (\mb_{i-1}^{\mathbf{s}})^\top ]^\top$, which satisfy $\|\ma - {\mq} {\mb}\|_{\mathrm{F}} \le \tau$ with high probability.

Solving \cref{1st opt prob} is equivalent to solving the following ``dual'' formulation:
\begin{align}
\label{2nd opt prob}
\max\ |\mathbf{r}| \quad\text{subject to}\quad \big\| \me_{i-1} + \mq_{i-1}^{\mathbf{r}}\mb_{i-1}^{\mathbf{r}} \big\|_{\mathrm{F}} \le \tau.
\end{align}
This equivalence follows immediately from the fact that $|\mathbf{r}|=b-|\mathbf{s}|$ and
\begin{align*}
\me_{i-1} + \mq_{i-1}^{\mathbf{r}}\mb_{i-1}^{\mathbf{r}}
&=
\me_{i-1} + \left( \mq_{i-1}\mb_{i-1} - \mq_{i-1}^{\mathbf{s}}\mb_{i-1}^{\mathbf{s}} \right) \\
&=
\left( \me_{i-1} + \mq_{i-1}\mb_{i-1} \right) - \mq_{i-1}^{\mathbf{s}}\mb_{i-1}^{\mathbf{s}} \\
&=
\me_{i-2} - \mq_{i-1}^{\mathbf{s}}\mb_{i-1}^{\mathbf{s}}
.
\end{align*}

\noindent
\textbf{Computational Challenge:} 
A naive approach for solving \cref{2nd opt prob} is to start with $\|\me_{i-1}\|_\mathrm{F}$ and \textit{increase} it by sequentially adding rank-$1$ outer products. However, forming the intermediate matrices explicitly is computationally expensive, and our method would lose its matrix-free advantage.

Instead, we only need to track how the Frobenius norm changes when we add rank-$1$ outer products from the final block $\mq_{i-1} \mb_{i-1}$. Because the columns of $\mq_{i-1}$ are orthonormal, adding or subtracting a rank-1 component only changes the overall Frobenius norm by the magnitude of the row vector in $\mb_{i-1}$. This is stated as the following theorem:
\begin{theorem} \label{thm:rr_identity}
Let $\mq_j$ and $\mb_j$ be the $j$-th blocks computed by \cref{alg:adapt_fro} (in exact arithmetic), and let $\me_{j}$ be the corresponding residual as defined in \cref{eq:ei}. Then, for any index set $\mathbf{r}\subset\{1,\ldots,b\}$, we have the following relation:
\begin{equation*}
\big\| \me_{j} + \mq_{j}^{\mathbf{r}}\mb_{j}^{\mathbf{r}} \big\|_{\mathrm{F}}^{2} = \|\me_{j}\|_{\mathrm{F}}^{2} + \|\mb_{j}^{\mathbf{r}}\|_{\mathrm{F}}^{2}.
\end{equation*}
\end{theorem}
\begin{proof}
By the construction of the QB factorization, the residual $\me_{j}$ is orthogonal to the column space of $\mq_{j}$. Because $\mq_{j}^{\mathbf{r}}$ is a submatrix of $\mq_{j}$ with orthonormal columns, the columns of $\me_j$ are orthogonal to the columns of $\mq_{j}^{\mathbf{r}}\mb_{j}^{\mathbf{r}}$. The identity then follows immediately from the matrix Pythagorean theorem.
\end{proof}

\vspace{1em}
Consequently, applying \cref{thm:rr_identity} to the final block, the optimization problem in \cref{2nd opt prob} reduces to selecting the largest subset of rows from $\mb_{i-1}$ whose squared Frobenius norms can be added to the squared residual norm without violating $\|\me_{i-1}\|_{\mathrm{F}}^{2} + \|\mb_{i-1}^{\mathbf{r}}\|_{\mathrm{F}}^{2} \le \tau^{2}.$ To summarize, we do not need to form any explicit matrices.

\vspace{1em}
\begin{center}
   \fbox{\parbox{0.65\linewidth}{We only need to compute the row norms of matrix $\mb_{i-1}$!
  }}
\end{center}
\vspace{1em}

To retain as few rank-1 outer products as possible, we sort the row norms of matrix $\mb_{i-1}$ in ascending order; start with $\|\me_{i-1}\|_\mathrm{F}^2$; and {increase} it by sequentially adding the row norms (squared) to just cross $\tau^2$. This is summarized in \cref{rr - F norm}.
Of course, we do not have the exact $\|\me_{i-1}\|_\mathrm{F}$ (and $\|\ma\|_\mathrm{F}$) but instead use their randomized estimators in \cref{eq:est_me,eq:est_ma}.

\begin{algorithm}
\caption{\texttt{Prune\_Rank\_Fro} - Rank pruning in the Frobenius norm}
\label{rr - F norm}
\begin{algorithmic}[1]
\REQUIRE
Matrix $\mb$ with $b$ rows, error (indicator) $\alpha$, and tolerance $\tau$

\ENSURE Index set $\mathbf{s} \subseteq \{1, 2, \dots, b\}$ of columns/rows to select.

\item[]

\STATE Compute $\mb$'s row norms $\{w_1, w_2, \ldots, w_b\}$.

\STATE Sort $\{w_1, w_2, \ldots, w_b\}$ in ascending order such that
$
w_{\pi(1)} \leq w_{\pi(2)} \leq \cdots \leq w_{\pi(b)}
$
for some permutation $\pi$.

\STATE $\mathbf{r} \gets \emptyset$ \algcomment{Initialize the set of redundant indices}
\FOR{$j = 1, \dots, b$}
    \STATE $\alpha \gets \sqrt{\alpha^2 + w_{\pi(j)}^2 }$
    \IF{$\alpha \ge \tau$}
        \STATE \textbf{break} \algcomment{Adding another row violates the tolerance}
    \ENDIF
    \STATE $\mathbf{r} \gets \mathbf{r} \cup \{\pi(j)\}$ \algcomment{Safely discard this index}
\ENDFOR

\STATE $\mathbf{s} \gets \{1, 2, \dots, b\} \setminus \mathbf{r}$ \algcomment{Retain all non-discarded indices}
\RETURN $\mathbf{s}$
\end{algorithmic}
\end{algorithm}

\subsection{Adjoint-free Adaptive Range Finder} \label{sec:adjoint_free_frob}

For applications where only the orthonormal basis $\mq$ is required and the projection $\mb = \mq^\top\ma$ is not explicitly needed, we present a variant of \cref{alg:adapt_fro} that requires only the forward operator $\ma$. By eliminating the need for the adjoint operator $\ma^\top$, this variant is both more broadly applicable and computationally faster. This task is commonly referred to as the \emph{range finder} problem~\cite{halko2011finding, martinsson2020randomized}.

In \cref{alg:adapt_fro}, the matrix $\mb$ is utilized in two specific locations: (i) the computation of the implicit residual $\my_i$ via the term $\mb_k \mO_i$, and (ii) the rank-pruning step which requires the row norms of $\mb_{i-1}$. We address these two requirements without the adjoint operator as follows.

First, we observe that the term $\mb_k \mO_i$ ($k=1,2,\ldots, i-1$) can be evaluated by shifting the parentheses to the right:
\begin{equation} \label{eq:bi}
\mb_k \mO_i = (\mq_k^\top \ma) \mO_i = \mq_k^\top (\ma \mO_i) = \mq_k^\top \mz_i 
\end{equation}
By computing $\mq_k^\top \mz_i$, we obtain the necessary projection using only the already-computed sample matrix $\mz_i$.

Second, the rank-pruning procedure typically requires the row norms of the final block $\mb_{i-1}$ to determine which columns of $\mq_{i-1}$ are redundant. In the adjoint-free case, we propose using the row norms of the sketch $\tilde{\mb}_{i-1} = \mb_{i-1} \mO_i$ as a proxy. According to \cref{thm:expect thm - Y_i}, the squared row norms of $\tilde{\mb}$ are unbiased Hutchinson estimators of the squared row norms of $\mb_{i-1}$. Specifically, for the $j$-th row of the final block $\mb_{i-1}$, denoted by $\mathbf{b}_j$, the pruning algorithm relies on the estimator $\hat{w}_j = \|\mathbf{b}_j \mO_i\|_2^2$. While we have established that $\mathbb{E}[\hat{w}_j] = \|\mathbf{b}_j\|_2^2$, the reliability of the pruning decision depends on the variance of this estimator, which is given by $\text{Var}(\hat{w}_j) = \frac{2}{b} \|\mathbf{b}_j\|_2^4$.
This variance decreases as the block size $b$ increases, meaning that larger block sizes not only improve BLAS-3 efficiency but also enhance the precision of the adjoint-free rank refinement. In practice, even with moderate block sizes (e.g., $b=16$), the probability of significantly misidentifying the ``low-energy'' rows is small. Furthermore, since the pruning step is only applied to the final block, any slight over-retention of columns due to estimator variance results only in a tiny increase in the final rank.

We summarize our adjoint-free method in \cref{alg:adjoint_free_frob}.

\begin{algorithm}
\caption{\texttt{randQB\_AF\_Fro}: Adaptive, adjoint-free randomized range finder in Frobenius norm}
\label{alg:adjoint_free_frob}
\begin{algorithmic}[1]

\REQUIRE Matrix $\ma \in \mathbb{R}^{m \times n}$, (absolute) tolerance\footnotemark $\varepsilon > 0$, and block size $b$.
\ENSURE Matrix $\mq$ with orthonormal columns such that $\|\ma - \mq\mq^\top\ma\|_{\mathrm{F}} \leq \varepsilon$ with high probability.

\item[]

\FOR {$i=1,2,\ldots$}
\STATE $\mathbf{\Omega}_i \gets \texttt{randn}(n, b)/\sqrt{b}$ \algcomment{normalized Gaussian block}
    \STATE $\mathbf{Z}_i \gets \mathbf{A}\mathbf{\Omega}_i$ \algcomment{Forward matvec only}
    \STATE $\my_i \gets \mz_i - \sum_{k = 1}^{i-1}\mq_k\left(\mq_k^\top \mz_i\right)$ \algcomment{Implicit application via \eqref{eq:bi}}
    \IF{$\|\my_i\|_{\mathrm{F}} \leq \varepsilon$}
        \STATE \textbf{break} 
        \algcomment{Last computed block is $\mq_{i-1}$}
    \ENDIF

    \STATE $\mathbf{Q}_i \gets \texttt{orth}(\mathbf{Y}_i)$
    \algcomment{$\mathbf{B}_i$ is not computed}
    \STATE $\mq_i \gets \texttt{orth}\left(\mq_i- \sum_{k=1}^{i-1}\mq_k\left(\mq_k^\top\mq_i\right)\right)$ \label{line: re-orth} \algcomment{Re-orthogonalization}
\ENDFOR

\STATE $\tilde{\mb}_{i-1} \gets \mq_{i-1}^\top\mz_i$ \label{line:compute-t} 
\algcomment{$\tilde{\mb}_{i-1}$ is a sketch of $\mb_{i-1}$ via \cref{eq:bi}}

\STATE $\textbf{s} \gets \texttt{Prune\_Rank\_Fro}(\tilde{\mb}_{i-1}, \, \|\my_i\|_{\mathrm{F}},\, \varepsilon)$

\STATE $\mq \gets [\mq_1 \; \ldots \; \mq_{i-2} \; \mq_{i-1}^\textbf{s}]$
\algcomment{$\mq_{i-1}^\textbf{s}$ contain the column subset \textbf{s} in $\mq_{i-1}$}

\end{algorithmic}
\end{algorithm}
\footnotetext{See \cref{rmk:relative tolerance} for the modification for relative tolerance.}

\begin{remark}[Relation to  an Existing Randomized Range Finder]
Compared to the standard fixed-tolerance range finder \cite[Algorithm 13]{martinsson2020randomized}, \cref{alg:adjoint_free_frob} introduces two enhancements: (i) a relative tolerance can be implemented in a straightforward mechanism via \cref{rmk:relative tolerance}, and (ii) a matrix-free rank pruning step that produces a more compact basis with negligible overhead. That reference proposes the basic adjoint-free range finder but reports no numerical results for it; our numerical study (\cref{sec:numerical}) fills this gap.
\end{remark}

\section{Adaptive Matrix-Free Algorithms in Spectral Norm} \label{sec: 4}

We now extend the matrix-free framework to the spectral norm (operator 2-norm). In this setting, the objective is to determine a target rank $k$ dynamically such that the following approximation holds:
\begin{equation} \label{eq:aqb_spec}
\| \mathbf{A} - \mathbf{Q}\mathbf{B} \|_2 \le \tau,
\end{equation}
where $\mathbf{Q} \in \mathbb{R}^{m \times k}$ has orthonormal columns and $\mathbf{B} = \mathbf{Q}^\top\mathbf{A} \in \mathbb{R}^{k \times n}$. The tolerance $\tau$ is a user-defined threshold, provided as either an absolute value $\tau = \varepsilon$ or a relative value $\tau = \varepsilon \|\mathbf{A}\|_2$.

As in the Frobenius-norm case, we seek to construct $\mq$ and $\mb$ incrementally using only matrix-vector products with $\ma$ and $\ma^\top$. However, the spectral norm presents a unique challenge: How to estimate the spectral norm of the residual matrix given a pair of factors $\mq$ and $\mb$. To address this, we resolve the following key algorithmic challenges:

\begin{enumerate}

\item \textbf{Residual Estimation (\cref{sec:err_ind_spec}):} We introduce a ``look-ahead'' randomized error indicator. By utilizing the fact that the QB framework naturally targets the dominant singular spectrum, the norm of the most recently computed block $\mathbf{B}_i$ provides an efficient estimate for the residual norm of the preceding approximation.

\item \textbf{Block Size Selection (\cref{sec:rr_spec}):} We analyze strategies for choosing the block size parameter to maximize computational efficiency while preventing severe overestimation of the numerical rank.

\item \textbf{Adjoint-Free Variant (\cref{sec:adjoint_free_spec}):} We introduce a forward-only modification that eliminates the need for adjoint queries ($\vy \mapsto \ma^\top \vy$) when only the basis $\mq$ is required.

\end{enumerate}

\subsection{Randomized Error Indicator and Matrix-Free Algorithm}
\label{sec:err_ind_spec}

The primary difficulty in an adaptive spectral-norm QB algorithm is that the stopping criterion depends on $\|\mathbf{E}\|_2$, where $\mathbf{E} = \mathbf{A} - \mathbf{Q} \mathbf{B}$ for a given pair of factors $\mq$ and $\mb$. We observe that at each iteration, the blocked QB method, i.e., \cref{alg:blocked_QB}, generates a basis $\mathbf{Q}_i$ whose span approximates the subspace spanned by dominant left singular vectors of the residual $\mathbf{E}_{i-1}$. If the block size $b$ is sufficiently large (e.g., $b=16$), then, with probability almost 1, the error $\| \me_{i-1} - \mq_i \mq_i^\top  \me_{i-1} \|_2$ is close to the minimum error in rank-$(b-p)$ approximation~\cite[Section 10]{halko2011finding}, where $p$ is the oversampling parameter (say $p=5$ or $p=10$).

Consequently, the matrix 
\[
\mathbf{B}_i = \mathbf{Q}_i^\top \mathbf{A} = \mathbf{Q}_i^\top \left( \mathbf{E}_{i-1} + \sum_{k=1}^{i-1} \mq_k \mb_k \right) = \mathbf{Q}_i^\top \mathbf{E}_{i-1}
\]
 captures the action of the residual on its most significant subspace, where the first equality is by the definition of $\mathbf{B}_i$, the second equality is by the definition of $\mathbf{E}_{i-1}$ in \cref{eq:ei}, and the last equality holds by the orthogonality of the basis matrices $\mathbf{Q}_i$.

\vspace{1em}
\begin{center}
\fbox{\parbox{0.9\linewidth}{Therefore, we use the spectral norm of the projected matrix as a proxy for the true error:
    \begin{equation*}
        \|\mb_{i}\|_2
        \approx
        \|\me_{i-1}\|_2.
    \end{equation*}
  }}
\end{center}
\vspace{1em}

We summarize our adaptive, matrix-free algorithm in \cref{alg:adapt_spec}.

\begin{remark}[Relative Tolerance] \label{rmk:relative_spec}
To adapt the stopping criterion for a relative tolerance $\tau = \varepsilon \|\mathbf{A}\|_2$, we utilize the first computed block as an initial estimate of the operator norm: $\|\mathbf{B}_1\|_2 \approx \|\mathbf{A}\|_2$. This allows the algorithm to remain entirely matrix-free even when the scale of $\mathbf{A}$ is unknown \textit{a priori}.
\end{remark}

\begin{remark}[Computing the spectral norm] 
To evaluate $\|\mathbf{B}_i\|_2$ efficiently, we avoid a full SVD. Since $\mathbf{B}_i \in \mathbb{R}^{b \times n}$ is a short-and-wide matrix, its spectral norm is the square root of the dominant eigenvalue of $\mathbf{B}_i \mathbf{B}_i^\top \in \mathbb{R}^{b \times b}$. Given that $b$ is small, a few steps of the power method suffice to determine this norm to sufficient accuracy with a cost of only $\mathcal{O}(nb)$ flops. 
\end{remark}

\begin{algorithm}

\caption{\texttt{randQB\_MF\_Spec}: Adaptive, matrix-free low-rank approximation in spectral norm}
\label{alg:adapt_spec}

\begin{algorithmic}[1]
\REQUIRE Matrix $\mathbf{A}\in \mathbb{R}^{m \times n }$, (absolute) tolerance\footnotemark $\varepsilon > 0$, and block size $b.$

\ENSURE Matrix $\mq$ with orthonormal columns and matrix $\mb$ such that $\|\mathbf{A} - \mq\mb\|_{2} \le \varepsilon$ with high probability.

\item[]

\FOR {$i=1,2,\ldots$}

    \STATE $\mathbf{\Omega}_i \gets \texttt{randn}(n, b)/\sqrt{b}$ \algcomment{normalized Gaussian block}
    \STATE $\mathbf{Z}_i \gets \mathbf{A}\mathbf{\Omega}_i$ \algcomment{$\mathbf{A}$ is accessed via matvec}
    \STATE $\my_i \gets \mz_i - \sum_{k = 1}^{i-1}\mq_k\left(\mb_k\mO_i\right)$
    \algcomment{Implicit residual application via \eqref{Y_i - full description}}
    \STATE $\mathbf{Q}_i \gets \texttt{orth}(\mathbf{Y}_i)$
    \algcomment{Thin QR decomposition}
    \STATE $\mathbf{Q}_i \gets \texttt{orth}\left(\mathbf{Q}_i - \sum_{k = 1}^{i-1}\mathbf{Q}_k\left(\mathbf{Q}_k^\top\mathbf{Q}_i\right)\right)$
    \algcomment{Re-orthogonalization}
    \STATE $\mathbf{B}_i \gets \mathbf{Q}_i^\top\mathbf{A}$ \algcomment{$\mathbf{A}$ is accessed via adjoint matvec}
    \IF{$\|\mb_i\|_{2} \leq \varepsilon$}
        \STATE \textbf{break} \algcomment{tolerance met; the look-ahead block $\mb_i$ is not retained}
    \ENDIF

\ENDFOR

\STATE $\left[\mU_{r}, \mathbf{W}_{r} \right] \gets \texttt{Prune\_Rank\_Spec}(\mb_{i-1}, \, \varepsilon)$
\algcomment{prune $\mb_{i-1}$, the last retained block}

\STATE $\mq \gets [\mq_1 \; \ldots \; \mq_{i-2} \; \left( \mq_{i-1} \mU_{r} \right) ]$

\STATE $\mb \gets \left[\mb_1^\top\; \cdots\; \mb_{i-2}^\top\; \mathbf{W}_r^\top \right]^\top$

\end{algorithmic}
\end{algorithm}
\footnotetext{See \cref{rmk:relative_spec} for the modification for relative tolerance.}

\subsection{Block Size Selection and Rank Pruning} \label{sec:rr_spec}

The block-size dilemma described in \cref{sec:rr} applies equally here: a large $b$ improves BLAS-3 throughput but lets the output rank $k = i \cdot b$ overshoot the minimal rank needed to meet the tolerance. We again prune the final retained block, but now in the spectral norm: instead of discarding rows by their norm, we truncate $\mathbf{B}_{i-1}$ according to its singular value distribution.

\vspace{1em}
\fbox{\parbox{0.85\linewidth}{We prune the final block $\mb_{i-1}$ to the fewest singular directions needed: retain those with singular value $\sigma_j(\mb_{i-1}) \ge \varepsilon$ and discard the rest, stripping away the excess rank while keeping the truncation error within $\varepsilon$.
  }}
\vspace{1em}

The change of criterion is dictated by the norm. In the Frobenius case, the orthogonality of $\mathbf{Q}_{i-1}$ makes the squared error additive over the rows of $\mathbf{B}_{i-1}$ (\cref{thm:rr_identity}), so pruning reduces to dropping rows of small norm. The spectral norm is not additive in this way; the quantity that governs it is the singular-value spectrum of the block $\mathbf{Q}_{i-1} \mathbf{B}_{i-1}$. Since $\mathbf{Q}_{i-1}$ has orthonormal columns, we compute the SVD of $\mathbf{B}_{i-1}$ and, invoking the Eckart--Young theorem, retain exactly the leading $r$ singular directions whose singular values reach the tolerance, discarding the trailing directions whose singular values---and hence whose contribution to the residual---fall below $\varepsilon$.

\vspace{1em}
\fbox{\parbox{0.85\linewidth}{Because the retained residual $\mathbf{E}_{i-1}$ is orthogonal to the final basis block $\mathbf{Q}_{i-1}$, the block-level truncation error combines with $\|\mathbf{E}_{i-1}\|_2$ in the Pythagorean sense rather than accumulating---so the pruned factors still meet the prescribed tolerance, up to the modest constant of \cref{thm:pruning_spec}.
  }}
\vspace{1em}

The procedure for rank pruning is summarized in \cref{alg:rr_spec}. Since $\mathbf{B}_{i-1}$ has dimensions $b \times n$, its SVD costs only $\mathcal{O}(nb^2)$ flops---negligible relative to the matvecs---yet it keeps the output rank close to the optimal numerical rank of the operator. The following theorem makes the tolerance guarantee precise.

\begin{algorithm}
\caption{\texttt{Prune\_Rank\_Spec}: Rank pruning in spectral norm}
\label{alg:rr_spec}
\begin{algorithmic}[1]
\REQUIRE Block $\mathbf{B} \in \mathbb{R}^{b \times n}$ and tolerance $\tau$.
\ENSURE Factors $\mathbf{U}_r$ and $\mathbf{W}_r = \mathbf{\Sigma}_r\mathbf{V}_r^\top$ such that $\|\mathbf{B} - \mathbf{U}_r \mathbf{W}_r\|_2 \le \tau$.

\item[]

\STATE $[\mathbf{U}, \mathbf{\Sigma}, \mathbf{V}] \gets \texttt{svd}(\mathbf{B}, \text{`econ'})$

\STATE Find indices $\mathbf{s} = \{j : \sigma_j(\mathbf{B}) \geq \tau \}$

\STATE $\mathbf{U}_r \gets \mathbf{U}(:, \mathbf{s}), \quad \mathbf{W}_r \gets \mathbf{\Sigma}(\mathbf{s}, \mathbf{s})\,\mathbf{V}(:, \mathbf{s})^\top$

\item[]
\algcomment{$\mathbf{W}_r$ formed by scaling the rows of $\mathbf{V}(:,\mathbf{s})^\top$}

\RETURN $\mathbf{U}_r, \mathbf{W}_r$

\end{algorithmic}
\end{algorithm}

\begin{theorem} \label{thm:pruning_spec}
Suppose \texttt{randQB\_MF\_Spec} (\cref{alg:adapt_spec}) is applied to $\mathbf{A} \in \mathbb{R}^{m \times n}$ with absolute tolerance $\varepsilon$ and terminates at iteration $i$, i.e., the look-ahead block satisfies $\|\mathbf{B}_i\|_2 \le \varepsilon$. Let $\tilde{\mathbf{Q}} = [\mathbf{Q}_1, \dots, \mathbf{Q}_{i-1}]$ and $\tilde{\mathbf{B}} = [\mathbf{B}_1^\top, \dots, \mathbf{B}_{i-1}^\top]^\top$ be the retained factors and $\mathbf{Q}, \mathbf{B}$ the pruned factors returned by the algorithm. Then
\begin{align*}
\|\mathbf{A} - \mathbf{Q}\mathbf{B}\|_2 \leq \sqrt{\eta^2 +1}\;\varepsilon,
\end{align*}
where $\eta$ is the constant from \cref{thm:arvind} applied to the residual $\mathbf{E}_{i-1} = \mathbf{A} - \tilde{\mathbf{Q}}\tilde{\mathbf{B}}$ defined in \eqref{eq:ei}.
\end{theorem}

\begin{proof}
Note that $[\mathbf{U}_r, \mathbf{W}_r] = \texttt{Prune\_Rank\_Spec}(\mathbf{B}_{i-1}, \varepsilon)$ with $\mathbf{W}_r = \mathbf{\Sigma}_r\mathbf{V}_r^\top$, where $r$ is the number of singular values of $\mathbf{B}_{i-1}$ that are at least $\varepsilon$, i.e., $\sigma_r(\mathbf{B}_{i-1}) \geq \varepsilon > \sigma_{r+1}(\mathbf{B}_{i-1})$. Let $\mathbf{N} = \mathbf{B}_{i-1} - \mathbf{U}_r \mathbf{W}_r$ be the error incurred by truncating the last block. Since the pruning replaces $\mathbf{Q}_{i-1}\mathbf{B}_{i-1}$ by $\mathbf{Q}_{i-1} (\mathbf{U}_r \mathbf{W}_r )$ and leaves the other blocks unchanged, we decompose the truncated residual as:
\begin{align*}
\mathbf{A} - \mathbf{Q}\mathbf{B}
=
(\mathbf{A} - \tilde{\mathbf{Q}}\tilde{\mathbf{B}}) + \mathbf{Q}_{i-1}(\mathbf{B}_{i-1} - \mathbf{U}_r \mathbf{W}_r)
=
\mathbf{E}_{i-1} + \mathbf{Q}_{i-1}\mathbf{N}.
\end{align*}
Since $\mathbf{Q}_{i-1}^\top \mathbf{E}_{i-1} = \mathbf{Q}_{i-1}^\top \mathbf{A} - \mathbf{Q}_{i-1}^\top \tilde{\mathbf{Q}}\tilde{\mathbf{B}} = \mathbf{B}_{i-1} - \mathbf{B}_{i-1} = \mathbf{0}$, the cross terms vanish in $(\mathbf{E}_{i-1} + \mathbf{Q}_{i-1}\mathbf{N})^\top(\mathbf{E}_{i-1} + \mathbf{Q}_{i-1}\mathbf{N}) = \mathbf{E}_{i-1}^\top\mathbf{E}_{i-1} + \mathbf{N}^\top\mathbf{N}$, and the sub-additivity of the largest eigenvalue gives:
\begin{align} \label{abs norm - spectral - triangle}
\|\mathbf{A} - \mathbf{Q}\mathbf{B}\|_2^2
=
\|\mathbf{E}_{i-1} + \mathbf{Q}_{i-1}\mathbf{N}\|_2^2
\leq
\|\mathbf{E}_{i-1}\|_2^2 + \|\mathbf{N}\|_2^2.
\end{align}
The look-ahead block satisfies $\mathbf{B}_i = \mathbf{Q}_i^\top\mathbf{A} = \mathbf{Q}_i^\top\mathbf{E}_{i-1}$, so applying \cref{thm:arvind} to $\mathbf{E}_{i-1}$ (with $\mathbf{B}_i$ as its projection) gives $\|\mathbf{E}_{i-1}\|_2 = \sigma_1(\mathbf{E}_{i-1}) \le \eta\,\sigma_1(\mathbf{B}_i) = \eta\,\|\mathbf{B}_i\|_2 \le \eta\varepsilon$. By the Eckart--Young theorem, the truncation error is $\|\mathbf{N}\|_2 = \sigma_{r+1}(\mathbf{B}_{i-1}) < \varepsilon$. Substituting these into \eqref{abs norm - spectral - triangle} yields:
\begin{align*}
\|\mathbf{A} - \mathbf{Q}\mathbf{B}\|_2^2 < \eta^2\varepsilon^2 + \varepsilon^2 = (\eta^2 + 1)\varepsilon^2.
\end{align*}
Taking the square root completes the proof.
\end{proof}

\begin{remark}
The constant $\eta$ from \cref{thm:arvind} is pessimistic. When the block size is sufficiently large so that the range finder captures the dominant singular subspace of $\mathbf{E}_{i-1}$ almost entirely, $\eta$ is close to $1$, and the attained error is very close to the prescribed $\varepsilon$. The numerical experiments in \cref{sec:numerical} confirm that the observed error tracks $\varepsilon$ closely, well within the bound of \cref{thm:pruning_spec}.
\end{remark}

\subsection{Adjoint-free Adaptive Range Finder}
\label{sec:adjoint_free_spec}

As in the Frobenius-norm case (\cref{sec:adjoint_free_frob}), many applications require only the orthonormal basis $\mq$ and not the projection $\mb = \mq^\top\ma$. We therefore present a variant of \cref{alg:adapt_spec} that accesses $\ma$ through the forward operator alone, eliminating the adjoint queries $\vy \mapsto \ma^\top\vy$.

In \cref{alg:adapt_spec} the adjoint is used in two places: (i) the look-ahead indicator $\|\mb_i\|_2$, where $\mb_i = \mq_i^\top\ma$, and (ii) the rank-pruning step, which truncates $\mb_{i-1}$ by its singular values. We replace both with forward-only quantities. Both replacements rest on the same principle that \cref{lem:subspace_embedding} justifies: a matrix's singular values are read from a Gaussian sketch of it. 

First, recall that $\mathbf{E}_{i-1} = \ma - \sum_{k=1}^{i-1}\mq_k\mq_k^\top\ma$ denotes the residual after the first $i-1$ blocks. Its random sketch $\mathbf{E}_{i-1}\mathbf{\Omega}_i$ can be assembled using only the forward action $\mz_i = \ma \mathbf{\Omega}_i$:
\begin{equation} \label{eq:af_resid_spec}
\mathbf{E}_{i-1}\mathbf{\Omega}_i
= 
\left(\ma - \sum_{k=1}^{i-1}\mq_k\mq_k^\top\ma\right)\mathbf{\Omega}_i
= 
\mz_i - \sum_{k=1}^{i-1}\mq_k\left(\mq_k^\top\mz_i\right)
=
\my_i,
\end{equation}
where we have used $(\mq_k^\top\ma)\mathbf{\Omega}_i = \mq_k^\top\mz_i$ to avoid forming $\mb_k$. As a result, the spectral norm of this sketch yields a look-ahead estimate of the residual,
\begin{equation} \label{eq:af_indicator_spec}
\|\my_i\|_2 = \| \mathbf{E}_{i-1}\mathbf{\Omega}_i \|_2 \approx \|\mathbf{E}_{i-1}\|_2 ,
\end{equation}
the spectral-norm analogue of the Frobenius indicator of \cref{sec:adjoint_free_frob}. It replaces $\|\mb_i\|_2$ as a crude stopping criterion and is estimated by a few steps of power iteration applied directly to $\my_i$.

Second, the pruning step of \cref{alg:rr_spec} truncates $\mb_{i-1}$ by its singular values, but $\mb_{i-1} = \mq_{i-1}^\top\ma$ is never formed here. Exactly as in the Frobenius-norm case \eqref{eq:bi}, we replace it with the forward sketch
\begin{equation} \label{eq:af_prune_spec}
\tilde{\mb}_{i-1} = \mq_{i-1}^\top\mz_i = \mq_{i-1}^\top\ma\,\mathbf{\Omega}_i = \mb_{i-1}\mathbf{\Omega}_i ,
\end{equation}
obtained from the current block's sample $\mz_i = \ma\mathbf{\Omega}_i$ at no additional matvec cost. Since $\mathbf{\Omega}_i^\top$ is, with high probability, a subspace embedding for the row space of $\mb_{i-1}$, \cref{lem:subspace_embedding} gives $\sigma_j(\tilde{\mb}_{i-1}) \approx \sigma_j(\mb_{i-1})$. Computing the SVD $\tilde{\mb}_{i-1} = \mathbf{U}\mathbf{\Sigma}\mathbf{V}^\top$ and retaining the $r$ left singular vectors with $\sigma_j(\tilde{\mb}_{i-1}) \geq \varepsilon$ gives the pruned basis $\mq = [\mq_1\;\ldots\;\mq_{i-2}\; \left(\mq_{i-1}\mathbf{U}_r \right)]$, the adjoint-free counterpart of \cref{alg:rr_spec}.

We summarize the procedure in \cref{alg:adjoint_free_spec}. A caveat is in order: guaranteeing a genuine $\varepsilon$-subspace embedding would require the block size $b$ to grow like $1/\varepsilon^2$, far more than is practical, so the sketch delivers only a same-order estimate of the singular values rather than a highly accurate one. We emphasize, then, that \eqref{eq:aqb_spec} is to be understood in an approximate sense: unlike the Frobenius-norm methods of \cref{sec: 3}, which certify the prescribed bound (with high probability), the spectral-norm algorithms do not guarantee \eqref{eq:aqb_spec} exactly, but instead keep the computed error $\|\mathbf{A} - \mathbf{Q}\mathbf{B}\|_2$ \emph{on the order of} the tolerance $\tau$.

\begin{algorithm}
\caption{\texttt{randQB\_AF\_Spec}: Adaptive, adjoint-free randomized range finder in spectral norm}
\label{alg:adjoint_free_spec}
\begin{algorithmic}[1]
\REQUIRE Matrix $\ma \in \mathbb{R}^{m \times n}$, (absolute) tolerance\footnotemark $\varepsilon > 0$, and block size $b$.
\ENSURE Matrix $\mq$ with orthonormal columns such that $\|\ma - \mq\mq^\top\ma\|_2 \lesssim \varepsilon$ with high probability.

\item[]

\FOR {$i = 1, 2, \ldots$}
    \STATE $\mathbf{\Omega}_i \gets \texttt{randn}(n, b)/\sqrt{b}$ \algcomment{normalized Gaussian block}
    \STATE $\mathbf{Z}_i \gets \mathbf{A}\mathbf{\Omega}_i$ \algcomment{Forward matvec only}
    \STATE $\my_i \gets \mz_i - \sum_{k=1}^{i-1}\mq_k\left(\mq_k^\top\mz_i\right)$ \algcomment{Adjoint-free residual sketch \eqref{eq:af_resid_spec}}
    \IF{$\|\my_i\|_2 \leq \varepsilon$}
        \STATE \textbf{break} \algcomment{tolerance met; the look-ahead block is not retained, see \eqref{eq:af_indicator_spec}}
    \ENDIF
    \STATE $\mathbf{Q}_i \gets \texttt{orth}(\my_i)$ \algcomment{Thin QR; $\mb_i$ is not computed}
    \STATE $\mathbf{Q}_i \gets \texttt{orth}\left(\mathbf{Q}_i - \sum_{k=1}^{i-1}\mq_k\left(\mq_k^\top\mathbf{Q}_i\right)\right)$ \algcomment{Re-orthogonalization}
\ENDFOR

\STATE $\tilde{\mb}_{i-1} \gets \mq_{i-1}^\top\mz_i$ \algcomment{forward sketch of $\mb_{i-1}$ via \eqref{eq:af_prune_spec}}
\STATE $[\mathbf{U}_r, \sim] \gets \texttt{Prune\_Rank\_Spec}(\tilde{\mb}_{i-1}, \, \varepsilon)$
\algcomment{prune sketch of the last retained block}

\STATE $\mq \gets [\mq_1 \; \ldots \; \mq_{i-2} \; \left( \mq_{i-1}\mathbf{U}_r \right) ]$
\end{algorithmic}
\end{algorithm}
\footnotetext{For a relative tolerance $\tau = \varepsilon\|\ma\|_2$, the operator norm $\|\ma\|_2$ is estimated from the forward sketches as described in \cref{rmk:relative_af_spec}.}

\begin{remark}[Relative tolerance] \label{rmk:relative_af_spec}
For a relative tolerance $\tau = \varepsilon\|\ma\|_2$, the adjoint-free method needs a matrix-free estimate of $\|\ma\|_2$. Each forward sketch $\mz_i = \ma\mathbf{\Omega}_i$ is itself a randomized sketch of $\ma$, so its spectral norm already supplies one: $\|\mz_i\|_2 \approx \|\ma\|_2$. Since the sketches $\mz_1,\ldots,\mz_i$ are all computed during the iteration, a sharper estimate is available at no extra matvec cost by concatenating them into $\mathcal{Z}_i = [\mz_1\;\cdots\;\mz_i] = \ma\,[\mathbf{\Omega}_1\;\cdots\;\mathbf{\Omega}_i]$, a sketch of $\ma$ with $i\,b$ columns, giving
\begin{equation*}
\|\ma\|_2 \approx \frac{1}{\sqrt{i}}\,\big\|[\mz_1\;\cdots\;\mz_i]\big\|_2 .
\end{equation*}
The estimate sharpens as $i$ grows, and the factor $1/\sqrt{i}$ rescales for the $i\,b$ columns of the concatenated sketch, reducing to $\|\mz_1\|_2 \approx \|\ma\|_2$ at $i=1$.
\end{remark}

\section{Numerical Results} \label{sec:numerical}

We evaluate the proposed matrix-free adaptive algorithms with two objectives: to show (1) that the adaptive stopping criteria accurately capture the target tolerance across a range of singular-value decay profiles, and (2) that the rank-pruning strategies yield a near-optimal output rank. All computations are performed in double precision.

\paragraph{Reproducibility} The MATLAB implementations of the proposed algorithms, together with the driver scripts that reproduce all numerical experiments and figures reported in this section, are publicly available at \url{https://github.com/SMART-Algebra/Adaptive-Matfree-Lowrank}.

\subsection{Experimental Setup and Baselines}\label{sec:setup}
We compare the proposed algorithms against the following baselines:

\begin{enumerate}
\item
\textbf{Frobenius-norm baseline (\texttt{randQB\_EI}~\cite{GULI}):} the SOTA adaptive QB method. Its error estimate relies on a ``top-down'' subtraction of squared Frobenius norms, which is prone to catastrophic cancellation in finite precision and typically limits its reliable tolerance to about $\mathcal{O}(\sqrt{\varepsilon_{\text{mach}}})$ (\cref{rmk:randQB_ei}). We show that our matrix-free variants remove this limitation, remaining stable down to machine precision.

\item
\textbf{Spectral-norm baseline (\texttt{randQB\_HMT}~\cite[Algorithm 4.2]{halko2011finding}):} the SOTA adaptive randomized range finder of \textbf{H}alko, \textbf{M}artinsson, and \textbf{T}ropp. It uses only the forward operator and returns only the orthonormal basis $\mq$, so it is adjoint-free. Its \emph{a posteriori} error estimate is provably reliable but pessimistic (\cite[Remark 4.1]{halko2011finding}), so it typically overestimates the rank.

\item
\textbf{Theoretical optimum (truncated SVD, t-SVD):} for any tolerance $\varepsilon$, the Eckart--Young theorem identifies the truncated SVD as the rank-$k$ approximation that minimizes the error in both the Frobenius and spectral norms. The SVD cannot be applied in the matrix-free setting without first forming the entire matrix, but we include it as a ``gold standard'' that measures the \emph{rank over-determination factor} of our algorithms --- how many extra dimensions the proposed randomized methods need to match the error of the optimal (but infeasible) SVD.
\end{enumerate}

\paragraph{Test matrices} We test our methods on matrices exhibiting a range of decay profiles, from rapid exponential decay (typical of smoothing operators) to slow algebraic decay (common in certain statistical datasets) to S-shape decay. Let $\ma = \mU\mathbf{\Sigma}\mv^\top$,  where $\mU, \mv\in \mathbb{R}^{1024 \times 1024}$ are random Gaussian matrices and $\mathbf{\Sigma} \in \mathbb{R}^{1024 \times 1024}$ is diagonal with entries defined below:

\begin{itemize}
    \item \textbf{Matrix 1:} $\sigma_i = \text{exp}(-i/2)$  \qquad \quad $1\leq i \leq 1024$ \hfill Exponential decay

    \item \textbf{Matrix 2:} $\sigma_i = 1/i^p$ \qquad \quad \quad \quad $1\leq i \leq 1024$ \hfill Polynomial decay

    \item \textbf{Matrix 3:}
    $\sigma_i =\begin{cases}
10^{-6(i-1)/399}, & 1 \le i \le 400, \\[6pt]
10^{\,a\, (i-400)^{-\alpha} + b}, & 401 \leq i \leq 1024,
\end{cases}$
\hfill S-shape decay \\
where $\alpha = 10^{-0.8} \approx 0.158$, and the constants $a$ and $b$ are fixed by requiring the tail to join the first piece continuously at $\sigma_{401} = 10^{-6}$ and to reach $\sigma_{1024} = 5\varepsilon_{\text{mach}}$; this gives $a \approx 14$ and $b \approx -20$.
\end{itemize}

The two exponents for Matrix~2 are chosen to stress-test the baselines, not to constrain our methods. For the Frobenius-norm experiments we take $p=4$, so that part of the spectrum falls below $\sqrt{\varepsilon_{\text{mach}}}$ and thereby exposes the precision wall of \texttt{randQB\_EI}. For the spectral-norm experiments we take $p=2$, so that the numerical ranks are relatively large and all methods require many iterations to converge, which sharpens the contrast between the proposed methods and \texttt{randQB\_HMT}. These choices are designed to highlight the differences between the methods and are not limitations of the proposed algorithms.

In all cases, the tolerance $\tau$ is specified as a relative threshold: $\tau = \varepsilon \|\mathbf{A}\|,$ for a wide range of $\varepsilon$, where $\|\cdot\|$ denotes either the Frobenius or spectral norm, depending on the algorithm under test. The relative criterion is more demanding than an absolute one, as the algorithm must additionally estimate $\|\mathbf{A}\|$ on the fly; our methods do so matrix-free from the same sketches (\cref{rmk:relative tolerance}), whereas \texttt{randQB\_EI} computes $\|\mathbf{A}\|_{\mathrm{F}}$ explicitly from $\mathbf{A}$ and is thus not matrix-free.

We use a block size of $b = 32$ throughout, except for Matrix~1, whose small numerical rank calls for $b = 16$ to keep the iteration counts comparable across matrices.
For each method we report the achieved (exact) relative error in the relevant norm. For \texttt{randQB\_EI} and the matrix-free variants (\texttt{randQB\_MF\_Fro}, \texttt{randQB\_MF\_Spec}), it is $\|\ma - \mq\mb\| / \|\ma\|$, where $\mq$ and $\mb$ are the  computed pair. The adjoint-free variants (\texttt{randQB\_AF\_Fro}, \texttt{randQB\_AF\_Spec}) and \texttt{randQB\_HMT} never form the projection $\mb = \mq^\top\ma$, so their error is evaluated as $\|\ma - \mq\mq^\top\ma\| / \|\ma\|$. For the t-SVD it is the optimal error of the best rank-$k$ approximation.

\subsection{Comparison in the Frobenius Norm}

We compare four methods: the truncated SVD (t-SVD; the optimal, but not matrix-free, baseline), \texttt{randQB\_EI} (SOTA, but not matrix-free), and the two proposed methods \texttt{randQB\_MF\_Fro} and \texttt{randQB\_AF\_Fro}. \Cref{fig:fro_exp,fig:fro_poly,fig:fro_sshape} report the results for the exponential, polynomial, and S-shaped spectra (Matrices~1--3), respectively. Each figure shows the computed rank versus the prescribed relative tolerance $\varepsilon$ (left) and the per-iteration error indicator of each method (right), together with a table of the computed rank $k$ and the achieved (exact) relative error in the Frobenius norm (evaluated as described in \cref{sec:setup}; tabulated at every other tolerance for brevity, while the panels use the full range).

\begin{figure}
    \centering
    \begin{subfigure}[b]{0.42\textwidth}
        \centering
        \includegraphics[width=\textwidth]{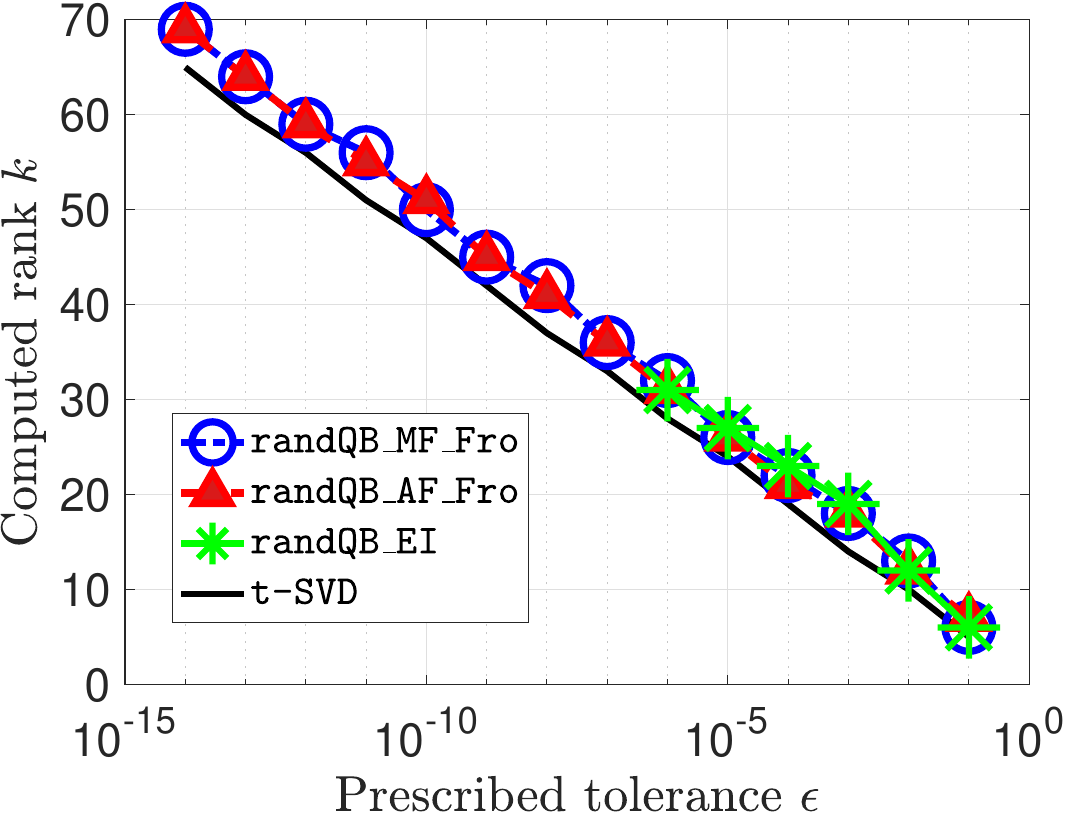}
    \end{subfigure}
    \hfill
    \begin{subfigure}[b]{0.42\textwidth}
        \centering
        \includegraphics[width=\textwidth]{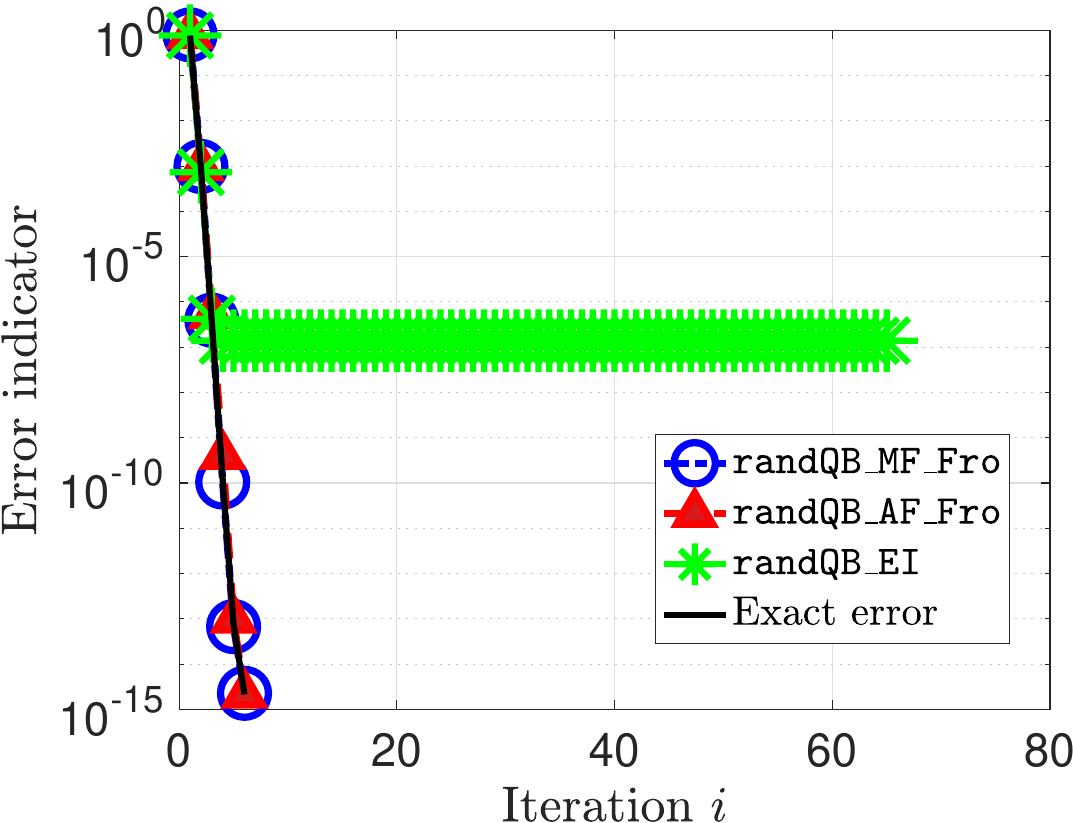}
    \end{subfigure}

    \vspace{1em}

    {\small
    \begin{tabular}{c|cc|cc|cc|cc}
    \hline
    $\epsilon$ & \multicolumn{2}{c}{t-SVD} & \multicolumn{2}{c}{randQB\_EI} & \multicolumn{2}{c}{randQB\_MF\_Fro} & \multicolumn{2}{c}{randQB\_AF\_Fro} \\
     & $k$ & rel.\ err & $k$ & rel.\ err & $k$ & rel.\ err & $k$ & rel.\ err \\ \hline
    1e-01 & 5 & 8.21e-02 & 6 & 8.79e-02 & 6 & 8.85e-02 & 7 & 5.59e-02 \\
    1e-03 & 14 & 9.12e-04 & 19 & 5.64e-04 & 18 & 7.86e-04 & 18 & 8.53e-04 \\
    1e-05 & 24 & 6.14e-06 & 27 & 7.86e-06 & 26 & 9.50e-06 & 26 & 8.44e-06 \\
    1e-07 & 33 & 6.83e-08 & \multicolumn{2}{c|}{\textcolor{red}{fail}} & 36 & 6.63e-08 & 36 & 9.51e-08 \\
    1e-09 & 42 & 7.58e-10 & \multicolumn{2}{c|}{\textcolor{red}{fail}} & 45 & 9.13e-10 & 45 & 8.28e-10 \\
    1e-11 & 51 & 8.42e-12 & \multicolumn{2}{c|}{\textcolor{red}{fail}} & 56 & 4.84e-12 & 55 & 5.04e-12 \\
    1e-13 & 60 & 9.40e-14 & \multicolumn{2}{c|}{\textcolor{red}{fail}} & 64 & 8.94e-14 & 64 & 8.13e-14 \\
    \hline
    \end{tabular}
    }

    \caption{Matrix~1 (exponential, fast spectral decay). Below its precision wall, \texttt{randQB\_EI}'s
    error indicator stops decreasing (right panel) and never meets the tolerance; those entries
    are marked \textcolor{red}{\texttt{fail}}.}
    \label{fig:fro_exp}
\end{figure}

\begin{figure}
    \centering
    \begin{subfigure}[b]{0.42\textwidth}
        \centering
        \includegraphics[width=\textwidth]{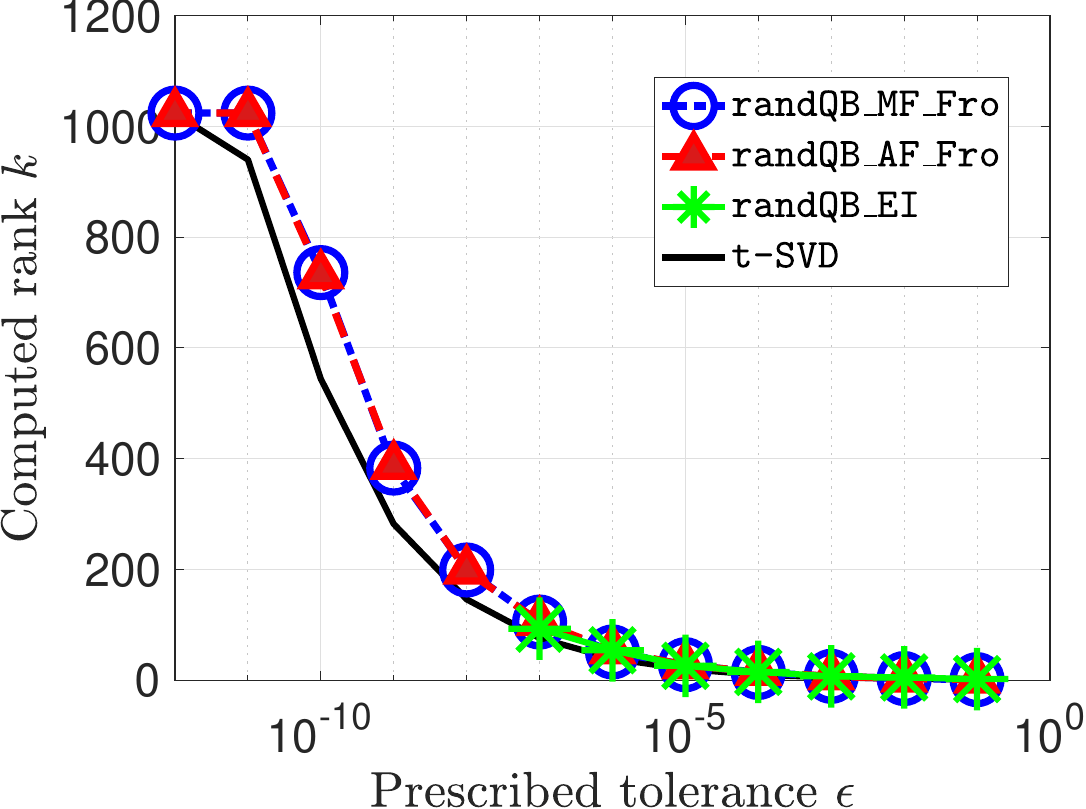}
    \end{subfigure}
    \hfill
    \begin{subfigure}[b]{0.42\textwidth}
        \centering
        \includegraphics[width=\textwidth]{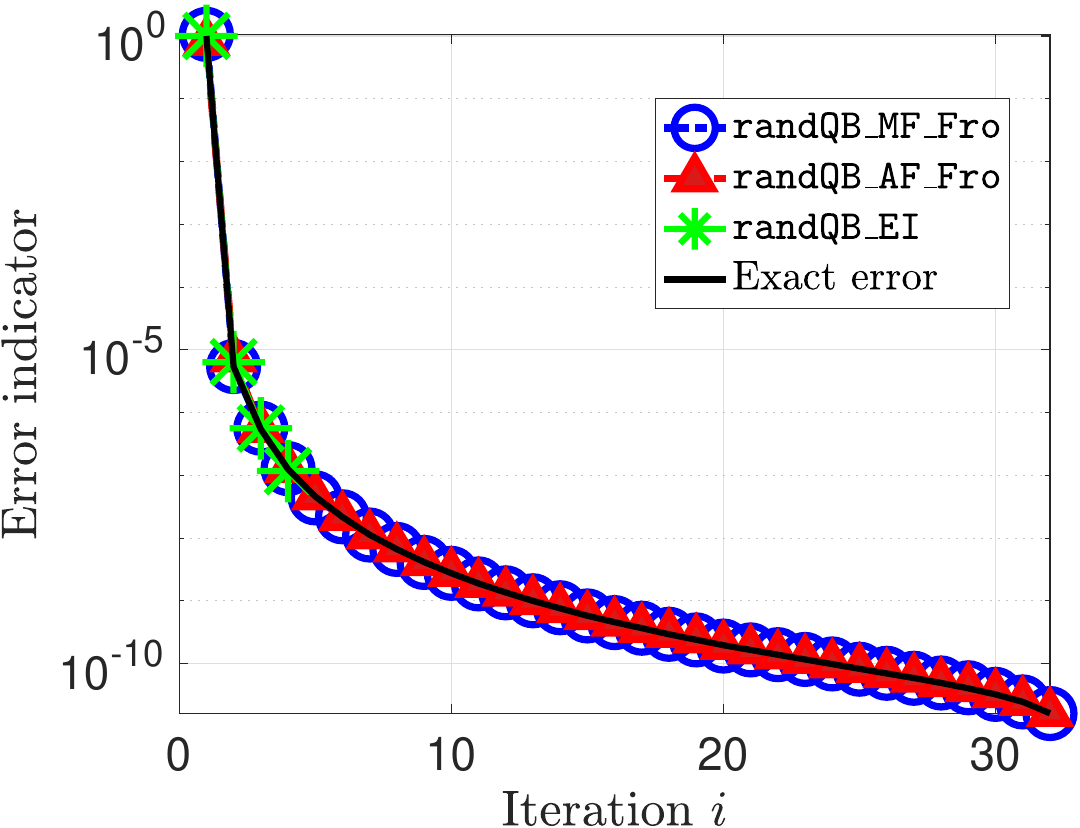}
    \end{subfigure}

    \vspace{1em}

    {\small
    \begin{tabular}{c|cc|cc|cc|cc}
    \hline
    $\epsilon$ & \multicolumn{2}{c}{t-SVD} & \multicolumn{2}{c}{randQB\_EI} & \multicolumn{2}{c}{randQB\_MF\_Fro} & \multicolumn{2}{c}{randQB\_AF\_Fro} \\
     & $k$ & rel.\ err & $k$ & rel.\ err & $k$ & rel.\ err & $k$ & rel.\ err \\ \hline
1e-01 & 1 & 6.37e-02 & 2 & 1.87e-02 & 1 & 8.11e-02 & 1 & 7.85e-02 \\
1e-03 & 5 & 9.32e-04 & 7 & 6.43e-04 & 7 & 7.18e-04 & 7 & 9.38e-04 \\
1e-05 & 20 & 9.64e-06 & 26 & 8.90e-06 & 28 & 9.76e-06 & 27 & 9.47e-06 \\
1e-07 & 76 & 9.63e-08 & \multicolumn{2}{c|}{\textcolor{red}{fail}} & 105 & 8.93e-08 & 105 & 9.40e-08 \\
1e-09 & 282 & 9.95e-10 & \multicolumn{2}{c|}{\textcolor{red}{fail}} & 383 & 9.98e-10 & 387 & 9.70e-10 \\
1e-11 & 940 & 9.92e-12 & \multicolumn{2}{c|}{\textcolor{red}{fail}} & 1024 & 1.51e-15 & 1024 & 1.49e-15 \\
\hline
    \end{tabular}
    }

    \caption{Matrix~2 (polynomial, slow spectral decay). Here the
    quantity $\|\ma\|_{\mathrm{F}}^2 - \|\mb\|_{\mathrm{F}}^2$ that \texttt{randQB\_EI} subtracts turns
    \emph{negative} below the precision wall, so its indicator is undefined and can never reach
    the tolerance (marked \textcolor{red}{\texttt{fail}}).
}
    \label{fig:fro_poly}
\end{figure}

\begin{figure}
    \centering
    \begin{subfigure}[b]{0.42\textwidth}
        \centering
        \includegraphics[width=\textwidth]{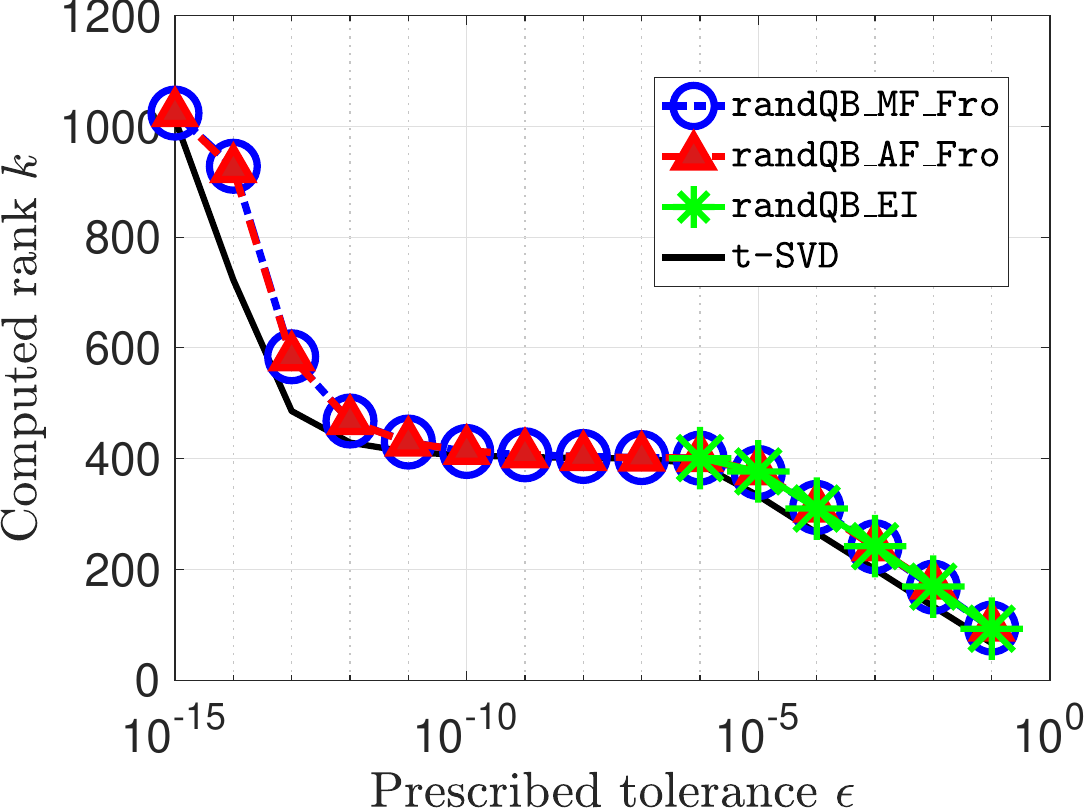}
    \end{subfigure}
    \hfill
    \begin{subfigure}[b]{0.42\textwidth}
        \centering
        \includegraphics[width=\textwidth]{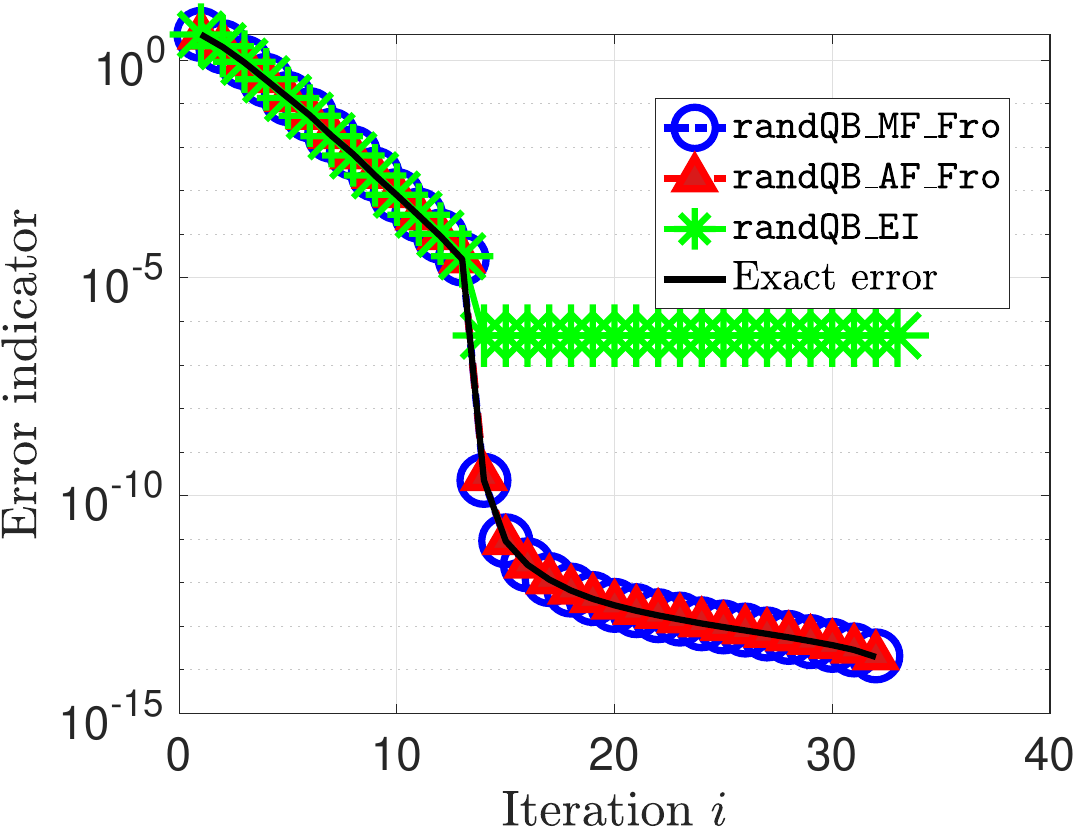}
    \end{subfigure}

    \vspace{1em}

    {\small
    \begin{tabular}{c|cc|cc|cc|cc}
    \hline
    $\epsilon$ & \multicolumn{2}{c}{t-SVD} & \multicolumn{2}{c}{randQB\_EI} & \multicolumn{2}{c}{randQB\_MF\_Fro} & \multicolumn{2}{c}{randQB\_AF\_Fro} \\
     & $k$ & rel.\ err & $k$ & rel.\ err & $k$ & rel.\ err & $k$ & rel.\ err \\ \hline
    1e-01 & 67 & 9.83e-02 & 93 & 9.91e-02 & 94 & 9.98e-02 & 95 & 9.47e-02 \\
    1e-03 & 200 & 9.83e-04 & 242 & 9.70e-04 & 241 & 9.66e-04 & 240 & 1.02e-03 \\
    1e-05 & 333 & 9.78e-06 & 377 & 9.73e-06 & 375 & 1.02e-05 & 377 & 9.94e-06 \\
    1e-07 & 401 & 9.17e-09 & \multicolumn{2}{c|}{\textcolor{red}{fail}} & 402 & 6.77e-08 & 402 & 2.99e-08 \\
    1e-09 & 403 & 4.95e-10 & \multicolumn{2}{c|}{\textcolor{red}{fail}} & 407 & 7.92e-10 & 408 & 8.01e-10 \\
    1e-11 & 412 & 9.18e-12 & \multicolumn{2}{c|}{\textcolor{red}{fail}} & 430 & 9.22e-12 & 429 & 9.85e-12 \\
    1e-13 & 486 & 9.81e-14 & \multicolumn{2}{c|}{\textcolor{red}{fail}} & 584 & 1.00e-13 & 583 & 1.04e-13 \\
    1e-15 & 1013 & 5.16e-15 & \multicolumn{2}{c|}{\textcolor{red}{fail}} & 1024 & 3.31e-15 & 1024 & 2.40e-14 \\
    \hline
    \end{tabular}
    }

    \caption{Matrix~3 (S-shaped spectral decay). Below its precision wall, \texttt{randQB\_EI}'s
    error indicator stops decreasing (right panel) and never meets the tolerance; those entries
    are marked \textcolor{red}{\texttt{fail}}.}
    \label{fig:fro_sshape}
\end{figure}

\paragraph{Accuracy and tolerance range}
Across all three spectra, the proposed \texttt{randQB\_MF\_Fro} matches \texttt{randQB\_EI} wherever the latter is reliable, while remaining accurate over a far wider range of tolerances. Its achieved relative error tracks the prescribed tolerance almost perfectly, staying just at or below $\varepsilon$ all the way down to machine precision. The adjoint-free variant \texttt{randQB\_AF\_Fro} is slightly less stable but still attains an achieved error of the same order as $\varepsilon$ throughout. In contrast, \texttt{randQB\_EI} is reliable only for $\varepsilon \gtrsim \mathcal{O}(\sqrt{\varepsilon_{\text{mach}}}) $; below that wall it breaks down (entries marked \textcolor{red}{\texttt{fail}} in the tables).

\paragraph{Rank efficiency}
The ranks produced by both proposed methods stay close to the optimal ranks of the t-SVD (typically within a handful of extra dimensions) so the matrix-free, randomized rank determination costs little. Unlike the t-SVD, however, the proposed methods access $\ma$ only through matrix-vector products and never form a dense factorization, which makes them far cheaper for large operators.

\paragraph{Error indicators and the breakdown of \texttt{randQB\_EI}}
The methods differ in how they monitor the residual error $\|\mathbf{A} - \mathbf{Q}\mathbf{B}\|_{\mathrm{F}}$. \texttt{randQB\_EI} uses the identity $\|\mathbf{E}\|_{\mathrm{F}}^2 = \|\mathbf{A}\|_{\mathrm{F}}^2 - \|\mathbf{B}\|_{\mathrm{F}}^2$. While exact in infinite precision, this ``top-down'' subtraction is highly susceptible to catastrophic cancellation once the residual is small relative to $\|\mathbf{A}\|_{\mathrm{F}}$, which (as noted in \cite[Theorem 3]{GULI}) restricts the reliable tolerance to roughly $\varepsilon \geq 2.7 \times 10^{-7}$ in double precision. Below this wall the method breaks down in one of two ways: either (i) the tracked quantity $\|\ma\|_{\mathrm{F}}^2 - \|\mb\|_{\mathrm{F}}^2$ turns \emph{negative}, so the indicator is undefined; or (ii) the indicator stops decreasing near $\mathcal{O}(\sqrt{\varepsilon_{\text{mach}}})$ and the stopping criterion is never met, so the iteration runs to its maximum number of steps and returns a full-rank factorization with no low-rank approximation. Which behavior occurs is driven by the random sketch rather than the spectrum: either can arise for any test matrix, and the same matrix may exhibit one or the other across runs; the figures report one representative failure each. In contrast, \texttt{randQB\_MF\_Fro} (and \texttt{randQB\_AF\_Fro}) use the randomized indicator
$
\|\mathbf{Y}\|_{\mathrm{F}} = \| ( \mathbf{A} - \mathbf{Q}\mathbf{B}) \mathbf{\Omega} \|_{\mathrm{F}} \approx \|\mathbf{A} - \mathbf{Q}\mathbf{B}\|_{\mathrm{F}},
$
where $\mathbf{\Omega}$ is a normalized Gaussian matrix (i.i.d.\ $\mathcal{N}(0,1/b)$ entries) with $b$ columns. Being a norm of a sketched residual, this indicator is inherently non-negative and free of the cancellation that plagues the energy-subtraction approach, so it reliably drives the iteration down to tolerances near machine precision.

\subsection{Comparison in the Spectral Norm}
We now extend the comparison to the spectral norm, evaluating four methods: the t-SVD (the optimal, but not matrix-free, baseline), the SOTA range finder \texttt{randQB\_HMT}~\cite[Algorithm~4.2]{halko2011finding}, and the two proposed methods \texttt{randQB\_MF\_Spec} and \texttt{randQB\_AF\_Spec}. \Cref{fig:spec_exp,fig:spec_poly,fig:spec_sshape} report the results for the exponential, polynomial, and S-shaped spectra (Matrices~1--3), respectively. Each figure shows the computed rank versus the prescribed relative tolerance $\varepsilon$ (left) and the per-iteration error indicator of each method at the smallest tolerance (right), together with a table of the computed rank $k$ and the achieved (exact) relative error in the spectral norm (evaluated as described in \cref{sec:setup}; tabulated at every other tolerance for brevity, while the panels use the full range).

\begin{figure}
    \centering
    \begin{subfigure}[b]{0.42\textwidth}
        \centering
        \includegraphics[width=\textwidth]{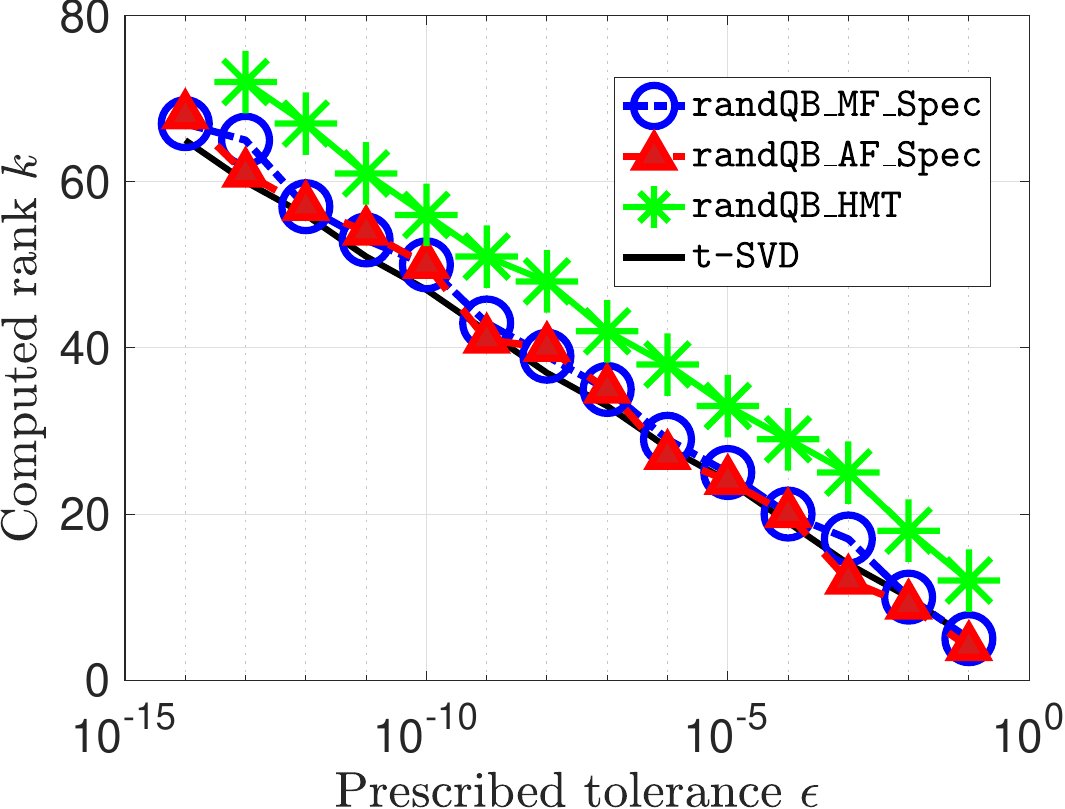}
    \end{subfigure}
    \hfill
    \begin{subfigure}[b]{0.42\textwidth}
        \centering
        \includegraphics[width=\textwidth]{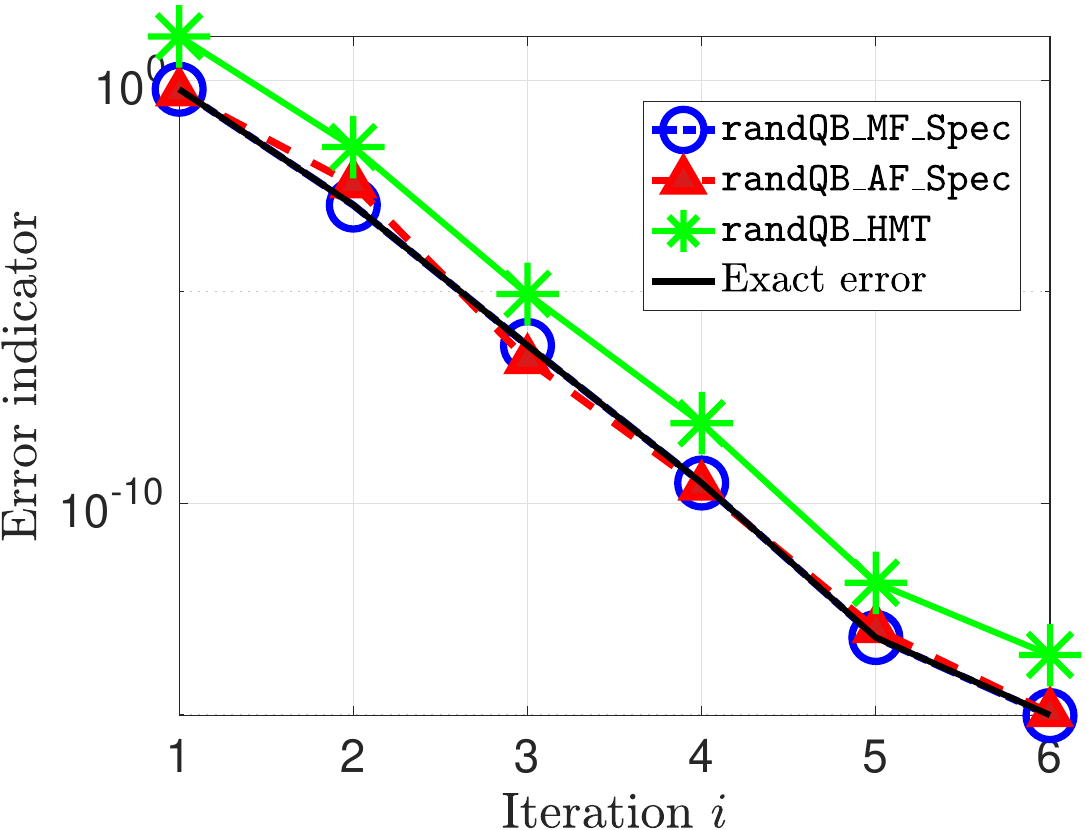}
    \end{subfigure}

    \vspace{1em}

    {\small
    \begin{tabular}{c|cc|cc|cc|cc}
    \hline
    $\epsilon$ & \multicolumn{2}{c}{t-SVD} & \multicolumn{2}{c}{randQB\_HMT} & \multicolumn{2}{c}{randQB\_MF\_Spec} & \multicolumn{2}{c}{randQB\_AF\_Spec} \\
     & $k$ & rel.\ err & $k$ & rel.\ err & $k$ & rel.\ err & $k$ & rel.\ err \\ \hline
    1e-01 & 5 & 8.21e-02 & 12 & 7.08e-03 & 5 & 8.21e-02 & 4 & 2.02e-01 \\
    1e-03 & 14 & 9.12e-04 & 25 & 1.43e-05 & 17 & 6.47e-04 & 12 & 5.24e-03 \\
    1e-05 & 24 & 6.14e-06 & 33 & 2.93e-07 & 25 & 4.59e-06 & 24 & 1.54e-05 \\
    1e-07 & 33 & 6.83e-08 & 42 & 5.22e-09 & 35 & 8.49e-08 & 35 & 6.00e-08 \\
    1e-09 & 42 & 7.58e-10 & 51 & 2.64e-11 & 43 & 9.43e-10 & 41 & 2.85e-09 \\
    1e-11 & 51 & 8.42e-12 & 61 & 5.73e-13 & 53 & 8.96e-12 & 54 & 6.34e-12 \\
    1e-13 & 60 & 9.36e-14 & 72 & 5.38e-15 & 65 & 4.68e-14 & 61 & 1.58e-13 \\
    \hline
    \end{tabular}
    }

    \caption{Matrix~1 (exponential, fast spectral decay).}
    \label{fig:spec_exp}
\end{figure}

\begin{figure}
    \centering
    \begin{subfigure}[b]{0.42\textwidth}
        \centering
        \includegraphics[width=\textwidth]{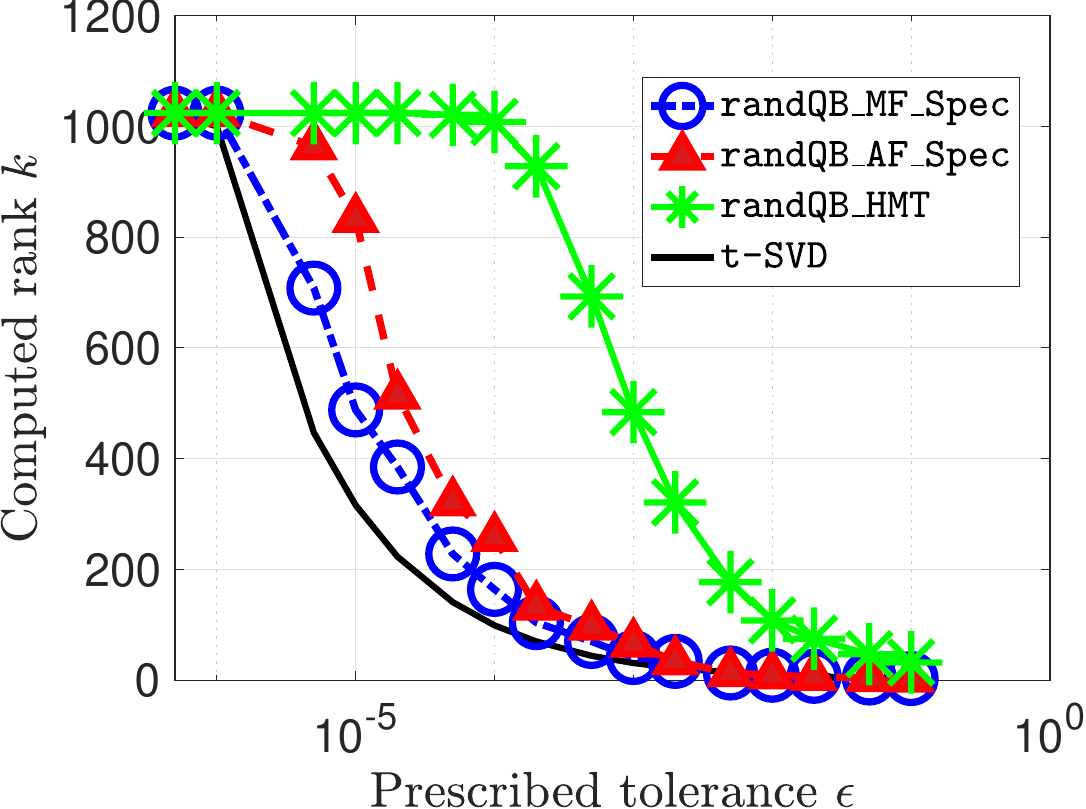}
    \end{subfigure}
    \hfill
    \begin{subfigure}[b]{0.42\textwidth}
        \centering
        \includegraphics[width=\textwidth]{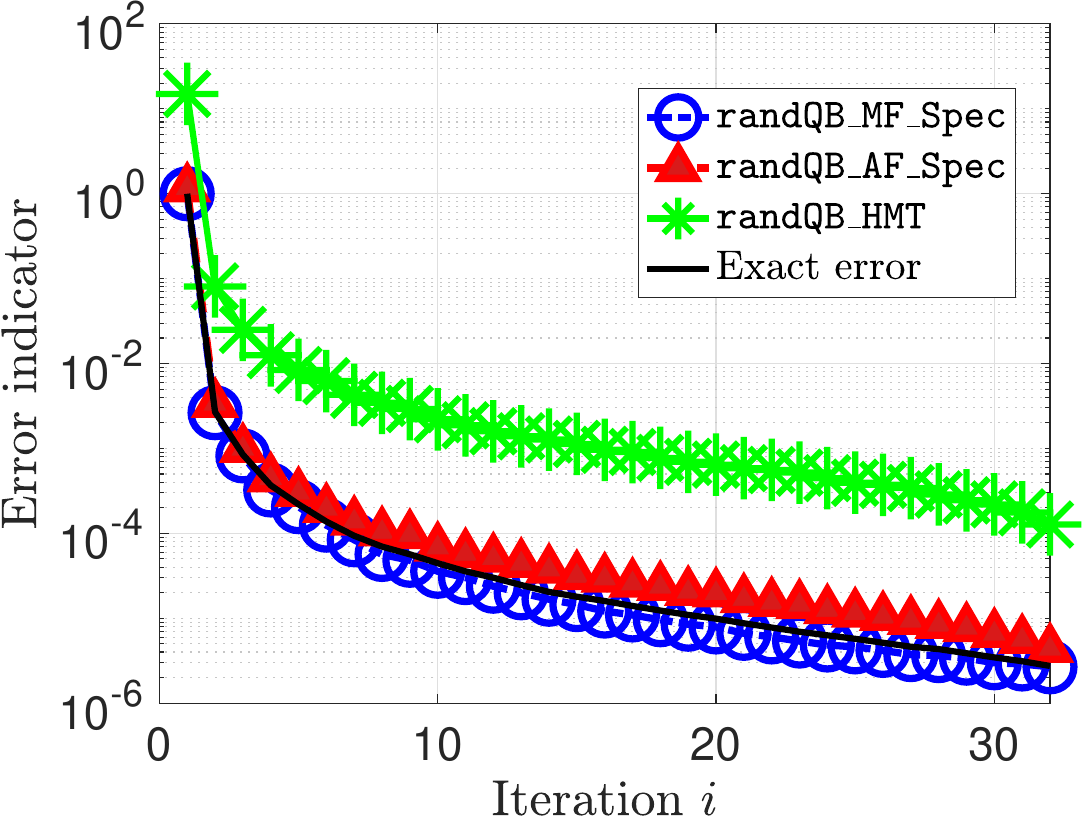}
    \end{subfigure}

    \vspace{1em}

    {\small
    \begin{tabular}{c|cc|cc|cc|cc}
    \hline
    $\epsilon$ & \multicolumn{2}{c}{t-SVD} & \multicolumn{2}{c}{randQB\_HMT} & \multicolumn{2}{c}{randQB\_MF\_Spec} & \multicolumn{2}{c}{randQB\_AF\_Spec} \\
     & $k$ & rel.\ err & $k$ & rel.\ err & $k$ & rel.\ err & $k$ & rel.\ err \\ \hline
    1e-01 & 3 & 6.25e-02 & 32 & 3.92e-03 & 3 & 6.25e-02 & 3 & 6.36e-02 \\
    2e-02 & 7 & 1.56e-02 & 75 & 5.57e-04 & 7 & 1.56e-02 & 6 & 2.21e-02 \\
    5e-03 & 14 & 4.44e-03 & 177 & 1.08e-04 & 13 & 5.16e-03 & 13 & 6.27e-03 \\
    1e-03 & 31 & 9.77e-04 & 484 & 1.53e-05 & 41 & 1.11e-03 & 66 & 7.66e-04 \\
    2e-04 & 70 & 1.98e-04 & 928 & 3.48e-06 & 104 & 2.44e-04 & 133 & 1.84e-04 \\
    5e-05 & 141 & 4.96e-05 & 1021 & 2.25e-06 & 228 & 6.25e-05 & 321 & 3.58e-05 \\
    1e-05 & 316 & 9.95e-06 & 1024 & 7.93e-16 & 488 & 1.42e-05 & 833 & 4.78e-06 \\
    1e-06 & 1000 & 9.98e-07 & 1024 & 6.50e-16 & 1024 & 1.15e-15 & 1024 & 1.51e-15 \\
    \hline
    \end{tabular}
    }

    \caption{Matrix~2 (polynomial, slow spectral decay). This slow decay inflates all the randomized ranks above the t-SVD optimum at the tighter tolerances.}
    \label{fig:spec_poly}
\end{figure}

\begin{figure}
    \centering
    \begin{subfigure}[b]{0.42\textwidth}
        \centering
        \includegraphics[width=\textwidth]{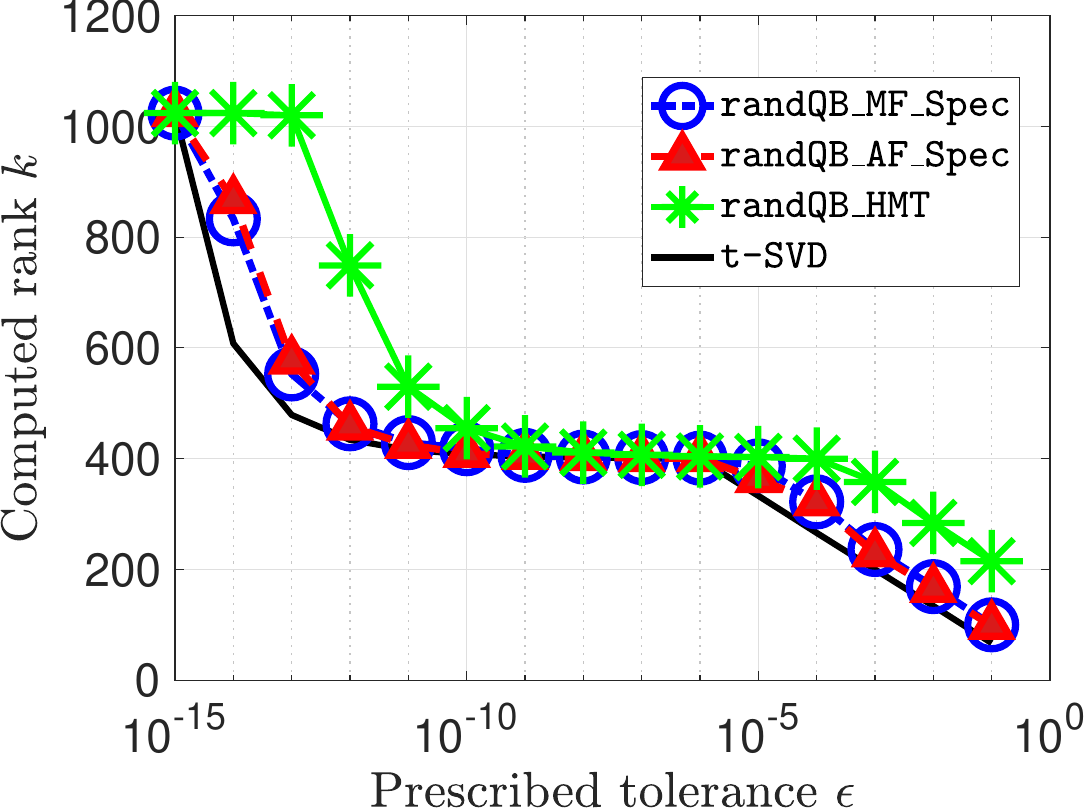}
    \end{subfigure}
    \hfill
    \begin{subfigure}[b]{0.42\textwidth}
        \centering
        \includegraphics[width=\textwidth]{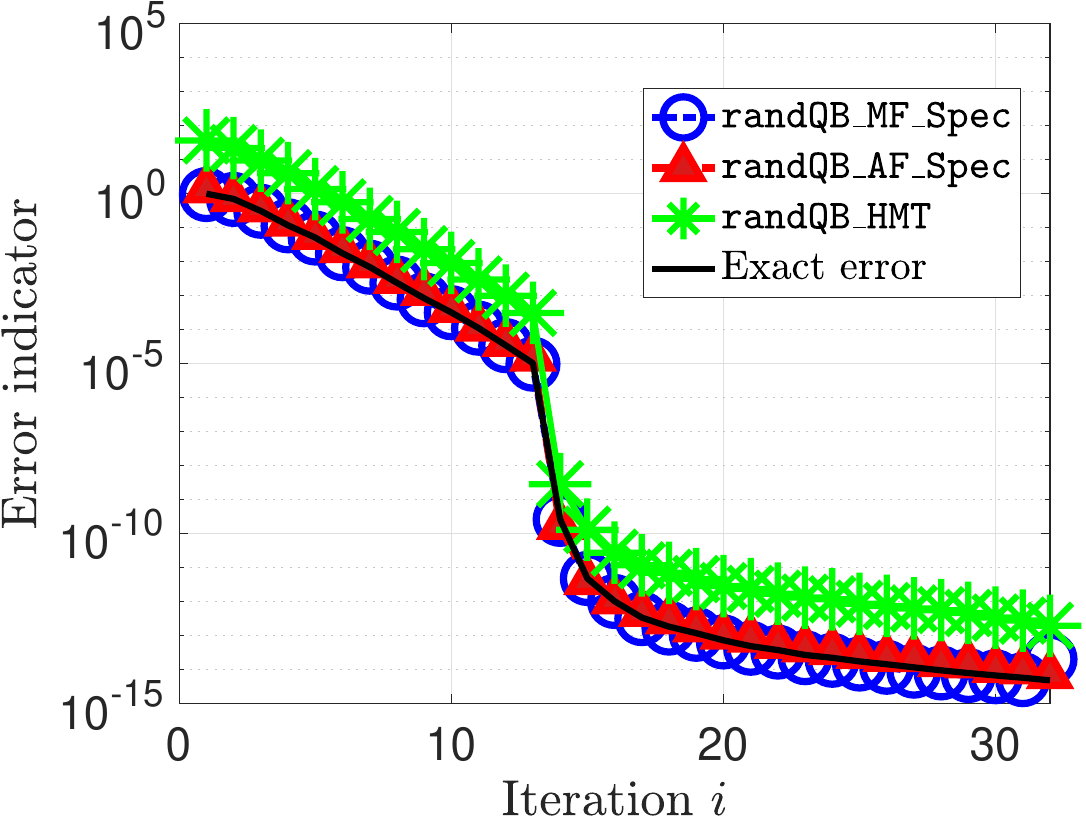}
    \end{subfigure}

    \vspace{1em}

    {\small
    \begin{tabular}{c|cc|cc|cc|cc}
    \hline
    $\epsilon$ & \multicolumn{2}{c}{t-SVD} & \multicolumn{2}{c}{randQB\_HMT} & \multicolumn{2}{c}{randQB\_MF\_Spec} & \multicolumn{2}{c}{randQB\_AF\_Spec} \\
     & $k$ & rel.\ err & $k$ & rel.\ err & $k$ & rel.\ err & $k$ & rel.\ err \\ \hline
    1e-01 & 67 & 9.83e-02 & 215 & 3.47e-03 & 100 & 9.06e-02 & 98 & 1.19e-01 \\
    1e-03 & 200 & 9.83e-04 & 358 & 3.29e-05 & 236 & 1.13e-03 & 229 & 2.00e-03 \\
    1e-05 & 333 & 9.83e-06 & 403 & 1.45e-07 & 387 & 8.52e-06 & 363 & 2.06e-05 \\
    1e-07 & 401 & 3.49e-08 & 407 & 1.98e-09 & 402 & 2.49e-08 & 401 & 2.83e-07 \\
    1e-09 & 404 & 7.03e-10 & 422 & 4.74e-11 & 406 & 8.80e-10 & 405 & 1.79e-09 \\
    1e-11 & 416 & 8.59e-12 & 530 & 2.59e-13 & 429 & 8.86e-12 & 425 & 1.73e-11 \\
    1e-13 & 479 & 9.74e-14 & 1020 & 3.38e-15 & 553 & 1.17e-13 & 577 & 9.78e-14 \\
    1e-15 & 1024 & 1.72e-14 & 1024 & 1.97e-14 & 1024 & 1.95e-14 & 1024 & 1.90e-14 \\
    \hline
    \end{tabular}
    }

    \caption{Matrix~3 (S-shaped spectral decay).}
    \label{fig:spec_sshape}
\end{figure}

\paragraph{Rank overestimation and comparison with \texttt{randQB\_HMT}}
A significant advantage of our method is its rank efficiency compared to the existing adaptive range finder. The existing method utilizes an error indicator that includes a safety constant $C = 10\sqrt{2/\pi}$. While this constant provides a rigorous probabilistic bound, it acts as a significant ``amplification'' factor that causes the algorithm to perceive the error as much larger than it actually is in many cases.

As illustrated in the left panels of \cref{fig:spec_exp,fig:spec_poly,fig:spec_sshape}, this pessimistic stopping criterion leads \texttt{randQB\_HMT} to consistently overestimate the rank required to satisfy the tolerance --- often inflating it to (near-)full rank at the tightest tolerances. In contrast, \texttt{randQB\_MF\_Spec} and \texttt{randQB\_AF\_Spec} avoid this over-determination, yielding numerical ranks that are far more compact --- close to the optimal t-SVD ranks for the rapidly-decaying spectra, and well below \texttt{randQB\_HMT}'s throughout.

\paragraph{Accuracy of the spectral error indicator}
A key feature of \texttt{randQB\_MF\_Spec} is the use of the look-ahead block norm $\|\mathbf{B}_i\|_2$ as a proxy for the total residual error. As shown in the right panels of \cref{fig:spec_exp,fig:spec_poly,fig:spec_sshape}, this indicator provides a remarkably tight estimate of the true residual error throughout the iteration process. The indicator remains reliable across many orders of magnitude.

\section{Conclusions}
We have introduced a family of adaptive, matrix-free randomized QB algorithms for high-precision low-rank approximation. Together, they address three bottlenecks of existing methods: excessive data passes, rank over-determination, and the precision wall inherent in Gramian-based error estimation.
Our numerical experiments demonstrate several advantages of the proposed framework:
\begin{itemize}
\item
\textbf{Robustness in the Frobenius Norm:} While the existing method \texttt{randQB\_EI} fails to satisfy tolerances below $\mathcal{O}(\sqrt{\varepsilon_{\text{mach}}})$, our matrix-free error indicator reliably tracks the residual error down to machine precision.

\item
\textbf{Efficiency in the Spectral Norm:}  By utilizing a look-ahead block norm indicator, we avoid the pessimistic constants found in an existing adaptive range finder, yielding significantly more compact ranks that closely match the optimal truncated SVD.

\item
\textbf{Reduced Data Movement:} When only the orthonormal basis is required, the adjoint-free variants dispense with the backward pass over the data, roughly halving the number of operator applications while degrading the approximation accuracy only slightly.

\end{itemize}

Several extensions are natural:

\paragraph{Tensor compression} Inserting the relative-tolerance range finder into the sequentially truncated HOSVD~\cite{vannieuwenhoven2012new,de2000multilinear} yields a matrix-free ST-HOSVD for the Tucker format. Standard randomized implementations estimate the per-mode truncation error with a Gramian-based indicator of the \texttt{randQB\_EI}~\cite{GULI} type, which is limited to tolerances of order $\mathcal{O}(\sqrt{\varepsilon_{\text{mach}}})$ (\cref{rmk:randQB_ei}); our indicators track the residual to machine precision, so the resulting ST-HOSVD overcomes this precision wall.

\paragraph{Hierarchical matrices} The construction of hierarchical ($\mathcal{H}$-)matrices, currently under active development, is a closely related target. Randomized black-box constructions build the representation from matrix--vector products, either by peeling off the off-diagonal blocks level by level~\cite{lin2011fast,martinsson2016compressing,Levitt_2024} or, more recently, by simultaneously compressing and factorizing the matrix~\cite{yesypenko2026randomized}. In all of these, the numerical rank fluctuates unpredictably across the many off-diagonal blocks, so fixed-rank range finders either over-allocate memory or lose fidelity, whereas the adaptive, fixed-tolerance methods of \cref{sec: 3} and \cref{sec: 4} set the rank block-by-block to a prescribed accuracy.

\paragraph{High-performance software} Because the algorithms are inherently blocked, they map naturally onto BLAS-3 kernels and modern parallel hardware, making a high-performance (GPU and distributed-memory) implementation a natural next step, building on recent HPC randomized-factorization efforts~\cite{heavner2023efficient, martinsson2019randutv, duersch2017randomized, dong2025robust, chen2025randomly}. A further direction is to extend these matrix-free techniques to streaming data.

\section*{Acknowledgments}
The authors used AI-based writing assistants to help edit language and improve the readability of the manuscript. All technical content was conceived, verified, and approved by the authors, who take full responsibility for the work.

\bibliographystyle{siamplain}

\end{document}